%% file: system-coupling.tex
\documentclass[a4paper,12pt]{article}
\RequirePackage[T1]{fontenc}
\RequirePackage{fix-cm}
\usepackage{amsmath,amsfonts,amsthm}
\usepackage{graphicx} % replaced for pdflatex
\usepackage[unicode, final, colorlinks=true]{hyperref} %replaced for pdflatex
\usepackage{cite}
\usepackage{caption}
\usepackage{mathtools}
\usepackage{newtxtext}
\usepackage[varvw]{newtxmath}
\usepackage{bm}
\usepackage[T1]{fontenc}
\usepackage[utf8]{inputenc}
\usepackage[dvipsnames]{xcolor}
\usepackage{tikz, pgfplots}
\usepackage{tikzscale}
\usepackage{booktabs}
\usepackage{multirow}
\usepackage{array}
\usepackage{csquotes}
\usepackage{accents}

\hypersetup{citecolor=teal, linkcolor=BrickRed, urlcolor=NavyBlue}
\pgfplotsset{compat=1.17}
\usepgfplotslibrary{groupplots}
\usepgfplotslibrary{fillbetween}
\usetikzlibrary{backgrounds}
\usetikzlibrary{arrows.meta}
\pgfplotsset{plot coordinates/math parser=false}
\usetikzlibrary{external}
\usepgfplotslibrary{patchplots}
\usetikzlibrary{shapes.geometric}
\usetikzlibrary{backgrounds}
\usetikzlibrary{intersections}
\usetikzlibrary{spy}

\providecommand{\keywords}[1]{\textit{Keywords:} #1}
\providecommand{\msc}[1]{\textit{2010 MSC:} #1}

\newtheorem{Def}{Definition}
\newtheorem{Lem}{Lemma}
\newtheorem{Prop}{Proposition}

\newtheorem{Rmk}{Remark}
\newtheorem{Ex}{Example}
\allowdisplaybreaks[4]

\newcommand{\R}{\mathbb{R}}
\newcommand{\N}{\mathbb{N}}
\newcommand{\Z}{\mathbb{Z}}

\newcommand{\eps}{\varepsilon}
\newcommand{\ddt}{\partial_t}
\newcommand{\ddx}{\partial_x}

\newcommand{\Psiq}{\Psi_Q}
\newcommand{\Psiqp}{{\tilde \Psi_Q}}

\begin{document}
\title{A central scheme for coupled hyperbolic systems}
\author{
  Michael~Herty$^{1}$ \and
  Niklas~Kolbe$^{1}$\footnote{Corresponding author} \and
  Siegfried~Müller$^1$
}

\date{
  \small
  $^1$Institute of Geometry and Applied Mathematics,\\ RWTH Aachen University, Templergraben 55,\\ 52062 Aachen, Germany\\
   \smallskip
   {\tt \{herty,kolbe,mueller\}@igpm.rwth-aachen.de} \\
   \smallskip
   \today
}
\maketitle

\begin{abstract}
A novel numerical scheme to solve coupled systems of conservation laws is introduced. The scheme is derived based on a relaxation approach and does not require information on the Lax curves of the coupled systems, which simplifies the computation of suitable coupling data. The coupling condition for the underlying relaxation system plays a crucial role as it determines the behavior of the scheme in the zero relaxation limit. The role of this condition is discussed, a consistency concept with respect to the original problem is introduced, well-posedness is analyzed and explicit, nodal Riemann solvers are provided. Based on a case study considering the $p$-system of gas dynamics a strategy for the design of the relaxation coupling condition within the new scheme is provided.

\par
\vspace{0.5em}
\noindent
\msc{35L65, 35R02, 65M08}
\par
\vspace{0.5em}
\noindent
\end{abstract}

\section{Introduction}\label{sec:intro}
Coupled systems of hyperbolic conservation or balance laws have been discussed in the mathematical literature to address various modeling questions e.g., regarding two-phase dynamics, multi-scale processes and networks, see e.g.,~\cite{bressan2014flowsnetwor}. The mathematical treatment of interface coupling has been motivated by boundary value problems~\cite{godlewski2004numerintercoupl,godlewski2005, ambroso2008coupllagran,chalons2008,dubois1988boundconditnonlin} or by applications to traffic flow on road networks~\cite{garavello2006traffflownetwor,MR2679594}, or gas dynamics in large-scale pipelines~\cite{banda2006coupleuler,MR2219276,brouwer2011gas,MR3335526}. It has been further expanded into a well-posedness result for coupled general gas dynamics, see e.g.  \cite{MR2438778,MR2441091,MR4175145}. Usually the coupling is imposed at a single spatial interface point as it has been done e.g., to model production systems~\cite{dapice2010model} and blood flow through systems of blood vessels~\cite{formaggia1999multismodelcirculsystem} to mention only a few.
\par 
 A main challenge has been the derivation of well-posed initial boundary value problems for coupled dynamics out of physically (or problem) induced conditions, see~\cite{bressan2014flowsnetwor} for a recent review. A well-known technique by now is the analytical concept of (half-)Riemann problems~\cite{herty2006coupl} or  solvers~\cite{garavello2016model}.
Moreover, the use of wave-front-tracking techniques yields well-posed coupled problems for a variety of mathematical fluid-type models, see e.g.~\cite{holden1995,MR1816648}.
This analytical tool has given rise to Riemann solver based numerical methods, which have been proposed e.g.~in~\cite{godlewski2004numerintercoupl,holden1995} and since then extended towards high-order methods~\cite{muller2015,bretti2006fast,borsche2016numer,banda2016numer}. Some high-order methods rely on the linearization of the imposed conditions such that the Lax-curves are obtained trivially, see e.g.~\cite{banda2016numer,borsche2014aderschemhigh}. We focus here on finite-volume schemes ---  finite-element based approaches, such as~\cite{egger2018}, require a different treatment of coupling conditions. Linearization techniques and Roe-type schemes have been used to avoid the explicit computation of eigenvalues in  two-phase problems, see e.g..~\cite{banda2015}. A linearization of the nonlinear dynamics at the coupling point is also at the core of an implicit (in space) numerical scheme~\cite{kolb2010}.
Recently using projection of the fluxes in the direction normal to the interface an extension of these techniques to the multi-dimensional case has been introduced in~\cite{herty2019couplcompreulerequat}. A related problem is given by conservation laws with discontinuous flux at a single point in space. This problem has applications, for instance, in porous media~\cite{gimse1992solutcauch}, sedimentation processes~\cite{diehl1995,burger2003}, traffic flow~\cite{holden1995} or in supply chain models~\cite{herty2013exist}. In the scalar, one-dimensional case, if conservation of mass and a modified entropy condition across the interface are imposed, solutions are  constructed~\cite{adimurthi2005optim} and numerically approximated ~\cite{mishra2017numermethodconser}. A numerical approximation based on relaxation~\cite{jin1995relaxschemsystem} of (scalar) discontinuous fluxes has been proposed in~\cite{karlsen2003relaxschemconser}. These schemes however, rely on analytical expression of Lax-curves which limits those techniques e.g., when treating phenomena like multi-phase flow. In various cases such curves are not available~\cite{hantke2018analysimulnew,hantke2019closurconditone} or not explicitly given~\cite{mueller2006riemanprobleuler}. 
\par
In this work a new approach is introduced that does not rely on Lax-curves at the interface. Therefore, building on our prior work on scalar equations on networks~\cite{herty2023centr}, the nonlinear systems are rewritten in the relaxation form proposed in~\cite{jin1995relaxschemsystem}. Due to the linearity of the relaxed systems the Lax-curves are linear spanned by the constant eigenvectors of the flux matrix. Based on the reformulation we derive an explicit scheme for coupled systems of conservation laws by taking a discretization to the zero relaxation limit as suggested also in \cite{jin1995relaxschemsystem}. A related approach has been followed in~\cite{borsche2018kinet}, where a kinetic relaxation model governs the coupling conditions. Another alternative is the construction of vanishing viscosity solutions, addressed in the schemes in~\cite{karlsen2017convergodunschem,towers2022explicfinitvolum}, that avoids the use of Lax-curves at the expense of formally treating parabolic systems. The system case that is discussed here requires a detailed analysis of the modified coupling condition imposed to the relaxation formulation, which we will consider.
\par 
The outline of this manuscript is as follows: in Section~\ref{sec:coupledrelaxation} we introduce the relaxation form for coupled systems of hyperbolic systems as well as a consistency concept for the corresponding coupling condition. In Section~\ref{sec:RS} we discuss suitable Riemann solvers of the coupled problem, discuss their well-posedness and provide an explicit form for the case of linear coupling conditions. Section~\ref{sec:schemes} is concerned with the construction of a finite-volume scheme for the coupled problem. In Section~\ref{sec:psystem}, we consider the $p$-system of gas dynamics in a case study and conduct numerical experiments for different choices of the relaxation coupling condition and then conclude providing a general strategy for the  construction of this condition.

\section{The coupled relaxation system}\label{sec:coupledrelaxation}
We consider a coupled system of hyperbolic conservation laws on the real line of the form 
\begin{subequations}\label{eq:hyperbolicsystem11}
\begin{align}
 \ddt U + \ddx F_1(U) &=0,\quad \text{in}(0, \infty) \times (-\infty, 0),\\
 \ddt U + \ddx F_2(U) &=0,\quad \text{in} (0, \infty) \times (0,\infty),
\end{align}
\end{subequations}
with vector valued state variable $U(t,x)\in \R^n$ and flux functions $F_1,F_2:\R^n\rightarrow \R^n$. At the interface located at $x=0$ a transition between the two flux functions takes place and we impose the coupling condition
\begin{equation}\label{eq:originalcoupling}
  \Psi_U(U(t, 0^-), U(t, 0^+)) =0 \qquad \text{for a. e. }t>0
\end{equation}
assuming a suitable mapping $\Psi_U:\R^n \times \R^n \rightarrow \R^m$ for $m\in\N$. To obtain a well-posed problem a sufficiently large number of suitable conditions is required and thus $m$ depends on the given flux functions. Moreover, we assume initial data given by the vector valued function $U^0(x)$, which is compatible with the coupling condition~\eqref{eq:originalcoupling}.

The relaxation system~\cite{jin1995relaxschemsystem} was introduced as a dissipative approximation of a hyperbolic system, recovering the original system in the so-called relaxation limit. In this work we use it as a tool to derive coupling information giving rise to a new method.
The relaxation system corresponding to~\eqref{eq:hyperbolicsystem11} employs the auxiliary state variable $V(t,x)\in \R^n$ and the relaxation rate $\varepsilon>0$ and reads 
\begin{subequations}\label{eq:relsyst11}
  \begin{align}
    \ddt U + \ddx V &= 0, && \text{in }(0, \infty) \times \R\setminus \{0 \}, \label{eq:relsyst11u}\\
    \ddt V + A_1 \, \ddx U &= \frac 1 \eps (F_1(U) - V),  && \text{in }(0, \infty) \times (-\infty, 0), \label{eq:relsyst11vl}\\
    \ddt V + A_2 \, \ddx U &= \frac 1 \eps (F_2(U) - V),  && \text{in }(0, \infty)\times(0, \infty), \label{eq:relsyst11vr}
  \end{align}
\end{subequations}
where $A_1$ and $A_2$ denote diagonal $n\times n$ matrices with positive diagonal entries $a_j^i$ for $j=1,\dots N$ and $i=1,2$. These matrices satisfy the subcharacteristic condition
\begin{equation}\label{eq:subchar}
  A_i - (DF_i(U))^2 \geq 0 \qquad i=1,2
\end{equation}
for all $U$. Here, $DF_i(U)$ refers to the Jacobian of $F_i$ in $U$ and the inequality indicates positive semi-definiteness. We note that both $U$ and $V$ depend on $\varepsilon$. Introducing the vector $Q=(U, V)$ we close system~\eqref{eq:relsyst11} using the relaxation coupling condition
\begin{equation}\label{eq:relaxationcoupling}
  \Psiq(Q(t, 0^-), Q(t, 0^+)) =0 \qquad \text{for a. e. }t>0.
\end{equation}
for a mapping $\Psiq:\R^{2n} \times \R^{2n} \rightarrow \R^{l}$, where $l\in\N$. The particular choice of $\Psi_Q$ and $m$ for given $\Psi_U$ is discussed in the following parts of this paper. Initial data for the relaxation system is derived from the original initial data by setting $Q^0=(U^0, V^0)$ such that
\begin{equation}
V^0(x)=F_1(U^0(x))\quad \text{if }x<0, \quad V^0(x)=F_2(U^0(x))\quad \text{if }x>0
\end{equation}
and compatibility of the initial data $Q^0$ with the coupling condition~\eqref{eq:relaxationcoupling} is assumed. 

It is well-known that the uncoupled relaxation system recovers the original system of conservation laws as well as $V=F(U)$ as $\varepsilon$ tends to zero. We are interested in this  relaxation limit at the interface as $\varepsilon \rightarrow 0$. 

\subsection{Continuous relaxation limit}\label{sec:continuouslimit}
The fact that the uncoupled relaxation system tends to a system of conservation laws in the relaxation limit, see~\cite{jin1995relaxschemsystem}, raises the question how the coupled systems and, in particular, the coupling conditions~\eqref{eq:originalcoupling} and~\eqref{eq:relaxationcoupling} relate. Here we address this question using asymptotic analysis and suppose that the solution $Q=(U, V)$ of the coupled relaxation system~\eqref{eq:relsyst11} endowed with the coupling condition~\eqref{eq:relaxationcoupling} can be written in terms of a Chapman--Enskog expansion~\cite{chapman1990mathemtheornonuniforgases}, i.e.,  
\begin{equation}
Q = \tilde Q^0 + \sum_{k=1}^\infty \varepsilon^k \tilde Q^k.
\end{equation}
As a solution of the relaxation system in the limit $\varepsilon \rightarrow 0$ on both half-axes, the state $\tilde Q^0=(\tilde U^0, \tilde V^0)$ has the property
\begin{equation}
  \tilde V^{0}(t, x) = F_1(\tilde U^{0}(t, x)) \text{ if }x<0, \qquad
  \tilde V^{0}(t,x) = F_2(\tilde U^{0}(t, x)) \text{ if }x>0
\end{equation}
and $\tilde U^0$ is a solution to the system of conservation laws on both half-axes. If we assume that $\Psi_Q$ is continuously differentiable, we obtain due to Taylor expansion
\begin{equation}
  \Psi_Q[Q(t, 0^-), Q(t, 0^+)] =   \Psi_Q\left[\tilde Q^0(t, 0^-), \tilde Q^0(t, 0^+)\right] + \mathcal{O}(\varepsilon)
\end{equation}
and, in particular, in the limit $\varepsilon \rightarrow 0$ we have
\begin{equation}
 \Psi_Q\left[\tilde Q^0(t, 0^-), \tilde Q^0(t, 0^+)\right] = 0.
\end{equation}
This observation motivates the following consistency condition.

\begin{Def}\label{def:consistency}
  The coupled relaxation system~\eqref{eq:relsyst11} is \emph{consistent} with the coupled system of conservation laws~\eqref{eq:hyperbolicsystem11} iff for a.~e. $t>0$ the corresponding coupling conditions~\eqref{eq:relaxationcoupling} and~\eqref{eq:originalcoupling} satisfy 
  \begin{equation}
    \Psi_U\left[U(t, 0^-), U(t, 0^+)\right] = 0 \quad \text{iff} \quad \Psiq \left[
    \begin{pmatrix} U(t, 0^-) \\
      f_1\left(U(t, 0^-) \right)
    \end{pmatrix},
    \begin{pmatrix} U(t, 0^+) \\
      f_2\left(U(t, 0^+) \right)
    \end{pmatrix}
  \right] = 0.
\end{equation}
\end{Def}
Assuming a smaller number of conditions for the system of conservation laws than for the relaxation system, i.e., $m<l$, this definition implies that consistency requires redundancies occurring in the relaxation limit due to~\eqref{eq:relaxationcoupling}.

\section{Riemann solvers for the coupled relaxation system}\label{sec:RS}
In this section we consider \emph{Riemann solvers} for the coupled relaxation system, which identify coupling data that both solve a half-Riemann problem and satisfy the coupling condition. They are an integral component of the finite volume schemes introduced in Section~\ref{sec:schemes}. We introduce the concept in Section~\ref{sec:halfRiemann}, discuss their well-posedness with respect to the coupling condition in Section~\ref{sec:generalcoupling}, derive an explicit form in the linear case in Section~\ref{sec:linearcoupling} and study their behavior in the relaxation limit in Section~\ref{sec:discretelimit}. 

\subsection{The half-Riemann problem}\label{sec:halfRiemann}
A finite volume scheme for a coupled problem such as~\eqref{eq:relsyst11} requires suitable coupling/boundary data, which is obtained by solving a \emph{half-Riemann problem}~\cite{dubois1988boundconditnonlin}. Later we will introduce the concept of \emph{Riemann solvers}, which identify coupling data that both solve the half-Riemann problem and satisfy a coupling condition. 

To discuss the half-Riemann problem in system~\eqref{eq:relsyst11} we require information about its Lax-curves. The relaxation system can be rewritten on both half-axes ($i=1$ if left and $i=2$ if right) as
\begin{equation}
  \ddt Q + S_i \, \ddx Q = \frac{1}{\varepsilon}
  \begin{pmatrix}
    0 \\
    F_i(U) - V
  \end{pmatrix}
\end{equation}
where $S_i\in \R^{2n \times 2n}$ refers to the diagonalizable block matrix
\begin{equation}\label{eq:diagonalization}
  S_i =
  \begin{pmatrix}
    0 & I \\
    A_i & O
  \end{pmatrix}
  =
  \begin{pmatrix}
    -(\sqrt{A_i})^{-1} & (\sqrt{A_i})^{-1} \\
    I & I
  \end{pmatrix}
  \begin{pmatrix}
    -\sqrt{A_i} & 0 \\
    0 & \sqrt{A_i}
  \end{pmatrix}
  \begin{pmatrix}
    -\frac 1 2 \sqrt{A_i} & \frac 1 2 I \\
    \frac 1 2 \sqrt{A_i} & \frac 1 2 I
  \end{pmatrix}
  =: R_i \Lambda_i R_i^{-1}
\end{equation}
By taking the root of each diagonal entry in $A$ we obtain $\sqrt{A_i}$, which is invertible and satisfies $\sqrt{A_i} \sqrt{A_i} = A_i$. The decomposition~\eqref{eq:diagonalization} shows that the eigenvalues of the system are given by
\begin{equation}\label{eq:eigenvalues}
\lambda_j = -\sqrt{a^i_j}, \qquad \lambda_{n+j} = \sqrt{a^i_j} \qquad \text{for }j=1,\dots,n. 
\end{equation}
The columns of the matrix $R_i$ include the (right) eigenvectors of the system that can be segmented as  
\begin{equation}
    R_i =
  \begin{pmatrix}
    -(\sqrt{A_i})^{-1} & (\sqrt{A_i})^{-1} \\
    I & I
  \end{pmatrix}
  =:
  \begin{pmatrix}
    R_i^- & R_i^+
  \end{pmatrix}
\end{equation}
so that $R_i^-\in\R^{2n \times n}$ includes the eigenvectors corresponding to the negative eigenvalues and $R_i^+\in\R^{2n \times n}$ the ones corresponding to the positive eigenvalues. Analogously we define the left eigenvectors as the rows of the matrix
\begin{equation}
  L_i \coloneq R_i^{-1}=
    \begin{pmatrix}
    -\frac 1 2 \sqrt{A_i} & \frac 1 2 I \\
    \frac 1 2 \sqrt{A_i} & \frac 1 2 I
    \end{pmatrix}
    =
  \begin{pmatrix}
    L_i^- \\ L_i^+
  \end{pmatrix},
\end{equation}
which can again be split with respect to the signs of the corresponding eigenvalues into the matrices $L_i^-$, $L_i^+\in\R^{n \times 2n}$.
The characteristic variables of the system are
\begin{equation}\label{eq:characteristics}
  \begin{pmatrix}
    W^-_i \\
    W^+_i
  \end{pmatrix}
  := R_i^{-1}Q = \frac 1 2
  \begin{pmatrix}
    V - \sqrt{A_i}U \\
    V + \sqrt{A_i}U
  \end{pmatrix}.
\end{equation}
Starting from a state $Q_0^*=(U_0^*, V_0^*)^T$ the Lax-curves of the systems are lines in the state space given by
\begin{equation}\label{eq:individuallaxcurves}
\mathcal{L}_{\lambda_j}(Q_0^*; \sigma_j^-)= Q_0^* - \sigma_j^-  (a^1_j)^{-1/2} \, e_j + \sigma_j^- e_{n+j}, \quad \mathcal{L}_{\lambda_{n+j}}(Q_0^*; \sigma_j^+)= Q_0^* + \sigma_j^+ (a^2_j)^{-1/2} e_{j} + \sigma_j^+ e_{n+j}
\end{equation}
for $j=1,\dots,n$ with $e_j \in \R^{2n}$ denoting the $j$-th unit vector and parameter $\sigma_j^\pm\in\R$. Moreover, we define the set
\begin{subequations}\label{eq:laxcurves}
\begin{equation}
 \mathcal{L}^-(Q_0^*) \coloneq \left\{ Q_0^* + R_1^- \Sigma^-:
    \Sigma^- \in \R^{n}
  \right\}
  =
  \left\{
    \begin{pmatrix}
      U_0^* -  (\sqrt{A_1})^{-1} \Sigma^- \\
      V_0^* + \Sigma^-
    \end{pmatrix}:
    \Sigma^- \in \R^{n}
  \right\}.
\end{equation}
which includes all states that connect to $Q_0^*$ by Lax curves with negative speeds, i.e.,  $\mathcal{L}_{\lambda_1}\dots\mathcal{L}_{\lambda_n}$. Conversely, the set of states connecting to $Q_0^*$ by Lax curves with positive speeds, i.e.,  $\mathcal{L}_{\lambda_{n+1}}\dots \mathcal{L}_{\lambda_{2n}}$, is 
\begin{equation}
  \mathcal{L}^+(Q_0^*) \coloneq \left\{ Q_0^* + R_2^+ \Sigma^+:
    \Sigma^+ \in \R^{n}
  \right\}
  =
  \left\{
    \begin{pmatrix}
      U_0^* + (\sqrt{A_2})^{-1} \Sigma^+ \\
      V_0^* + \Sigma^+
    \end{pmatrix}:
    \Sigma^+ \in \R^{n}
  \right\}.
\end{equation}
\end{subequations}
Note that for the following discussion the Lax-curves with negative speeds $\mathcal{L}_{\lambda_1},\dots,\mathcal{L}_{\lambda_n}$ in~\eqref{eq:laxcurves} as well as $\mathcal{L}^-$ in~\eqref{eq:individuallaxcurves} are defined for the relaxation system on the left half-axis, whereas the Lax-curves with positive speeds $\mathcal{L}_{\lambda_{n+1}},\dots,\mathcal{L}_{\lambda_{2n}}$ and $\mathcal{L}^+$ are defined for the relaxation system on the right half-axis. 

\begin{figure}
  \centering
  \input{input/wave-structure.tikz}
  \caption{Wave structure of the coupled relaxation system in the $x$-$t$-plane. The left trace data $Q_0^{-}$ are connected to the coupling data $Q_R$ by the Lax-curves of negative speeds, i.e., $\mathcal{L}_{\lambda_1}\dots\mathcal{L}_{\lambda_n}$ and the outgoing trace data $Q_0^{+}$ are connected to the coupling data $Q_L$ by the Lax-curves of positive speeds, i.e., $\mathcal{L}_{\lambda_{n+1}}\dots \mathcal{L}_{\lambda_{2n}}$. Coupling data on the incoming and outgoing edges are related by the coupling condition $\Psiq$.}\label{fig:waves}
\end{figure}
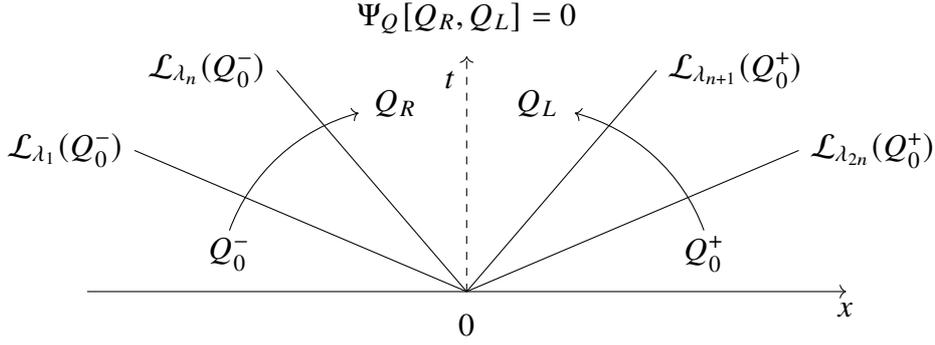

We consider a situation, where we are given discrete data next to the coupling interface from a finite volume scheme. We refer to this so-called trace-data by $Q_0^- = (U_0^-, V_0^-)$ left from the interface and $Q_0^+ = (U_0^+, V_0^+)$ right from the interface, respectively. We denote boundary/coupling data for the left half axis $(-\infty, 0)$ by $Q_R=(U_R, V_R)$. Analogously, we denote coupling/boundary data for the right half axis $(0, \infty)$ by $Q_L=(U_L, V_L)$. If $Q_R$ and $Q_L$ constitute a solution of the half-Riemann problem they need to satisfy
\begin{equation}\label{eq:laxcondition}
  Q_R \in \mathcal L^-(Q_0^-) \qquad \text{and} \qquad Q_L \in \mathcal L^+(Q_0^+).
\end{equation}
A Riemann solver (RS) for the relaxation system~\eqref{eq:relsyst11} corresponding to the coupling condition~\eqref{eq:relaxationcoupling} is a mapping
\begin{equation}\label{eq:RSdef}
\mathcal{RS}: \R^{2n \times 2n}\rightarrow \R^{2n \times 2n}, \qquad (Q_0^-, Q_0^+) \rightarrow (Q_R, Q_L)
\end{equation}
that assigns coupling data satisfying the coupling condition and solving the half-Riemann problem to given trace data. It is defined by condition~\eqref{eq:laxcondition} and
\begin{equation}\label{eq:psiq}
  \Psiq[Q_R, Q_L] = 0
\end{equation}
for the coupling function introduced in~\eqref{eq:relaxationcoupling}. We note that the well-posedness of the RS depends on the coupling condition. The relation between traces, coupling data and coupling condition in the relaxation system is illustrated in Figure~\ref{fig:waves}.

\subsection{Well-posedness}\label{sec:generalcoupling}
In this section we assume general coupling conditions of the form~\eqref{eq:relaxationcoupling} with $l=2n$, i.e., 
\begin{equation}\label{eq:psigeneral}
  \Psiq[Q_R, Q_L] = 0, \qquad \Psi_Q:\R^{2n}\times \R^{2n} \rightarrow \R^{2n}.
\end{equation}
To construct a corresponding RS we aim to identify coupling data that not only satisfies~\eqref{eq:psigeneral} but also solves the half-Riemann problem~\eqref{eq:laxcondition}. Condition~\eqref{eq:laxcondition} allows us to parameterize the coupling data as
\begin{equation}\label{eq:couplingdataparameterized}
  Q_R(\Sigma; Q_0) = Q_0^- + R^-_1 \Sigma^-, \qquad Q_L(\Sigma; Q_0) = Q_0^+ + R^+_2 \Sigma^+,
\end{equation}
using the notation
\begin{equation}\label{eq:Q0leftright}
    Q_0 = \left( Q_0^-, Q_0^+ \right), \quad
  \Sigma =
  \left( \Sigma^-, \Sigma^+\right)
  \end{equation}
and thus to introduce the parametrized form of the coupling function
\begin{equation}
  \Psiqp[\Sigma; Q_0] \coloneq \Psi_{Q}[ Q_R(\Sigma; Q_0),  Q_L(\Sigma; Q_0)], \qquad \Psiqp:\R^{2n}\times \R^{4n} \rightarrow \R^{2n}.
\end{equation}
Given this parametrization we can interpret the RS~\eqref{eq:RSdef} for coupling condition~\eqref{eq:psigeneral} as a routine that first identifies $\Sigma$ by solving the root problem $\Psiqp[\Sigma; Q_0]=0$ for given $Q_0$ and then outputs the data $Q_R(\Sigma; Q_0)$ and $Q_L(\Sigma; Q_0)$.
It is well posed iff for any trace data $Q_0$ the equation $\Psiqp[\Sigma; Q_0]=0$ has a unique solution $\Sigma = \Xi(Q_0)$. We state a result on its local regularity.
\begin{Lem}\label{lem:implicit}
Let $\Psiq$ be continuously differentiable. Suppose that for a fixed $Q_0\in\R^{4n}$ there is a unique $\Sigma\in \R^{2n}$ satisfying $\Psiqp[\Sigma; Q_0]=0$ and the Jacobian $D_\Sigma \Psiqp[\Sigma; Q_0]$ is invertible. Then there is an open ball $B_{Q_0}\subset \R^{4n}$ around $Q_0$ such that the RS~\eqref{eq:RSdef} with coupling condition~\eqref{eq:psigeneral} is well defined and continuously differentiable for all $\tilde Q_0=(\tilde Q_0^-, \tilde Q_0^+)^T\in B_{Q_0}$.
\end{Lem}
\begin{proof}
  Due to the implicit function theorem there are open balls $B_\Sigma$ and $B_{Q_0}$ around $\Sigma$ and $Q_0$, respectively, and a unique continuously differentiable mapping $\Xi: B_{Q_0} \rightarrow B_{\Sigma}$ such that
\begin{equation}\label{eq:implicitxi}
\Psiqp[\tilde \Sigma; \tilde Q_0] = 0 \quad \text{iff}\quad \tilde \Sigma = \Xi(\tilde Q_0) 
\end{equation}
for all $\tilde Q_0 \in B_{Q_0}$ and $\tilde \Sigma \in B_\Sigma$. The statement is obtained taking into account that the output of the RS~\eqref{eq:RSdef} is uniquely given by $Q_R(\Xi(\tilde Q_0); \tilde Q_0)$ and $Q_L(\Xi(\tilde Q_0); \tilde Q_0)$ for all $\tilde Q_0 \in B_{Q_0}$.
\end{proof}

A further result is obtained applying the contraction mapping theorem to a fixed point iteration of the form
\begin{equation}\label{eq:fpiteration}
  \Sigma^{n+1} = \Sigma^n - A(Q_0) \Psiqp[\Sigma^n; Q_0]
\end{equation}
with $A(Q_0)\in\R^{2n \times 2n}$ invertible.

\begin{Lem}\label{lem:banach}
  Let $\mathcal{S}\subset \R^{4n}$ and $\mathcal{P} \subset \R^{2n}$ denote open sets. Suppose that for any $Q_0 \in \mathcal{S}$ there is a regular matrix $A(Q_0)\in\R^{2n \times 2n}$ such that
  \begin{equation}\label{eq:fpselfmapping}
    T_{Q_0}(\Sigma) = \Sigma - A(Q_0) \Psiqp[\Sigma; Q_0] \in \mathcal{P}
  \end{equation}
  for all $\Sigma \in  \mathcal{P}$ and $T_{Q_0}$ is Lipschitz continuous with Lipschitz constant $K < 1$. 
 Then $\mathcal{RS}:\mathcal{S} \rightarrow \R^{2n}\times \R^{2n}$ defined through~\eqref{eq:laxcondition},~\eqref{eq:RSdef} and~\eqref{eq:psigeneral} is well defined.
\end{Lem}
\begin{proof}
  For $Q_0 \in \mathcal{S}$ the function $T_{Q_0}(\Sigma)$ maps from $\mathcal{P}$ to $\mathcal{P}$ due to~\eqref{eq:fpselfmapping} and is a contraction mapping as it is Lipschitz continuous with sufficiently small Lipschitz constant. Thus, by the contraction mapping theorem, the sequence~\eqref{eq:fpiteration} admits for any arbitrary $\Sigma^0\in \mathcal{P}$ a unique fixed point $\tilde \Sigma$ satisfying $\Psiqp[\tilde \Sigma; Q_0] =0$. This implies well-posedness of $\mathcal{RS}_{\Psiqp}$.
\end{proof}

\begin{Rmk}\label{rmk:matrixnorm}
  The Lipschitz continuity of $T_{Q_0}$ in Lemma~\ref{lem:banach} is implied by the stronger condition
    \begin{equation}\label{eq:fpderivativebound}
   \sup_{\Sigma \in \mathcal{P}} \|I - A(Q_0)D_\Sigma \Psiqp[Q_0; \Sigma] \| < 1.
 \end{equation}
\end{Rmk}
We note that Lemmata~\ref{lem:implicit} and~\ref{lem:banach} allow to freely choose the vector norm in $\R^n$ determining the induced matrix norm in Remark~\ref{rmk:matrixnorm}.

\subsection{Linear coupling conditions}\label{sec:linearcoupling}
In this section we discuss a special case of~\eqref{eq:psigeneral}, namely affine linear coupling functions of the form
\begin{equation}\label{eq:linearcoupling}
\Psiq[Q_R, Q_L] = B_R Q_R - B_L Q_L - P 
\end{equation}
for $B_R, B_L \in \R^{2n \times 2n}$ and $P \in \R^{2n}$. Under further conditions on the matrices we can provide explicit formulas of the coupling data  obtained by the corresponding RS~\eqref{eq:RSdef}. We introduce the block matrices
\begin{equation}\label{eq:truncatedR}
  \tilde R^- =
  \begin{pmatrix}
    R_1^- & 0
  \end{pmatrix}
  =
  \begin{pmatrix}
    - (\sqrt{A_1})^{-1} & 0 \\
    I & 0
  \end{pmatrix}
  , \qquad
    \tilde R^+ =
  \begin{pmatrix}
    0 & R_2^+
  \end{pmatrix}
  =
  \begin{pmatrix}
    0 &  (\sqrt{A_2})^{-1} \\
    0 & I
  \end{pmatrix}
\end{equation}
satisfying $\tilde R^- \Sigma = R_1^- \Sigma^-$ and $\tilde R^+ \Sigma = R_2^+ \Sigma^+$ assuming definition~\eqref{eq:Q0leftright}.

\begin{Prop}\label{prop:linearrs}
  Suppose the matrix $B_R \tilde R^- - B_L \tilde R^+$ is regular. Then the RS corresponding to the affine linear coupling function~\eqref{eq:linearcoupling} is well defined and yields the output
\begin{equation}\label{eq:affinelinearrsoutput}
    Q_R = Q_0^- + \tilde R^- \, \Sigma, \qquad Q_L = Q_0^+ + \tilde R^+ \, \Sigma,
  \end{equation}
  where $\Sigma$ is the unique solution of the linear system
  \begin{equation}\label{eq:sigmaexplicit}
    \left( B_R \tilde R^- - B_L \tilde R^+\right) \Sigma = P + B_L Q_0^+ - B_R Q_0^-.
  \end{equation}
\end{Prop}
\begin{proof} Using the parametrization~\eqref{eq:couplingdataparameterized} we rewrite the coupling condition as
\begin{align*}
    0 &= \Psiq[Q_R(\Sigma; Q_0), Q_L(\Sigma; Q_0)] = B_R \left( Q_0^- + \tilde R^- \Sigma \right) - B_L \left( Q_0^+ + \tilde R^+ \Sigma \right) - P,
\end{align*}
which implies~\eqref{eq:sigmaexplicit}.
\end{proof}

In case of block matrices the system~\eqref{eq:sigmaexplicit} can be solved using a block LU decomposion.
\begin{Rmk}\label{rmk:linearblockmatrices}
  Suppose the matrices $B_R$ and $B_L$ are given in block form as
  \begin{equation}
    B_R =
    \begin{pmatrix}
      B_R^{11} &  B_R^{12} \\[5pt]
      B_R^{21} &  B_R^{22}
    \end{pmatrix}, \qquad
    B_L =
       \begin{pmatrix}
      B_L^{11} &  B_L^{12} \\[5pt]
      B_L^{21} &  B_L^{22}
    \end{pmatrix}
  \end{equation}
  with each block being a matrix in $\R^{n \times n}$. The system matrix in~\eqref{eq:sigmaexplicit} then takes the form
  \begin{equation}
    B_R \tilde R^- - B_L \tilde R^+  =
    \begin{pmatrix}
      B_R^{12} -  B_R^{11}(\sqrt{A_1})^{-1} &   -B_L^{12} - B_L^{11} (\sqrt{A_2})^{-1} \\[5pt]
       B_R^{22} - B_R^{21}(\sqrt{A_1})^{-1} &  -B_L^{22} - B_L^{21} (\sqrt{A_2})^{-1}
    \end{pmatrix}
    \eqcolon
    \begin{pmatrix}
      B^{11} & B^{12} \\
      B^{21} & B^{22}
    \end{pmatrix}
    = B.
  \end{equation}
  Suppose now that $B^{11}$ and $S \coloneq B^{22}-B^{21}(B^{11})^{-1} B^{12}$ are regular, then the RS corresponding to the coupling condition~\eqref{eq:linearcoupling} is well defined and coupling data can be computed using the inverse system matrix
  \begin{equation}\label{eq:luinverse}
    B^{-1} =     \begin{pmatrix}
      (B^{11})^{-1} \left( I + B^{12} S^{-1} B^{21} (B^{11})^{-1}\right) &  -(B^{11})^{-1} B^{12}S^{-1} \\
      -S B^{21} (B^{11})^{-1} & S^{-1}
    \end{pmatrix}.
  \end{equation}
\end{Rmk}

If $B^{11}$ in Remark~\ref{rmk:linearblockmatrices} is not regular a block LU decomposition of $B$ and an inverse matrix similar to~\eqref{eq:luinverse} might still be computable after suitable rotation of the rows and columns.

Note that assuming the coupling condition~\eqref{eq:linearcoupling} the relaxation system~\eqref{eq:relaxationcoupling} is consistent (in the sense of Definition~\ref{def:consistency}) with system~\eqref{eq:hyperbolicsystem11} under the coupling condition
\begin{equation}\label{eq:linearcouplinginu}
  \Psi_U[U_R, U_L] = B_R
  \begin{pmatrix}
U_R \\ f_1(U_R)
  \end{pmatrix}
  - B_L
    \begin{pmatrix}
U_L \\ f_2(U_L)
    \end{pmatrix}
    - P = 0.
  \end{equation}
  
\begin{Ex}[Kirchhoff conditions]\label{ex:KirchhoffCoupling}
A canonical choice of coupling conditions for~\eqref{eq:relsyst11} are the ones due to Kirchhoff, which require equality of the flux left and right from the interface, in all equations of the system, i.e.,  
\begin{equation}\label{eq:kirchhoffcouplingcondition}
  \Psiq^K[Q_R, Q_L] = S_1 Q_R - S_2 Q_L = 0.
\end{equation}

By choosing the diagonal entries sufficiently large (such that~\eqref{eq:subchar} holds) we can achieve $A_1=A_2=:A$, which we assume in the following. Consequently, the quantities in~\eqref{eq:diagonalization},~\eqref{eq:characteristics} and~\eqref{eq:laxcurves} are independent of the half-axis considered and we therefore neglect the index $i$. Since $S$ is regular, the coupling condition~\eqref{eq:kirchhoffcouplingcondition} simplifies to
\begin{equation}\label{eq:couplingconditionsimplified}
  \Psi_Q^K[Q_R, Q_L] = Q_R - Q_L = 0
\end{equation}
and thus the coupling function is of the form~\eqref{eq:linearcoupling} for $B_R=B_L=I$ and $P=0$.

From~\eqref{eq:sigmaexplicit} we obtain
\[
  \begin{pmatrix}
      -\Sigma^- \\
      ~\Sigma^+
  \end{pmatrix}
  =
    \begin{pmatrix}
    R^- & R^+
    \end{pmatrix}^{-1}
    (Q_0^- - Q_0^+)
    =
    \begin{pmatrix}
      L^- \\
      L^+
    \end{pmatrix}
    (Q_0^- - Q_0^+)
  \]
  and consequently the coupling data is given by
  \[
    Q_R = Q_0^- - R^- L^-  (Q_0^- - Q_0^+), \quad Q_R = Q_0^+ + R^+ L^+  (Q_0^- - Q_0^+),
    \]
    or equivalently
\begin{subequations}\label{eq:couplingdata}
\begin{align}
  U_R = U_L &= \frac 1 2 (U_0^- + U_0^+) + \frac 1 2 (\sqrt{A})^{-1}(V_0^- - V_0^+), \\
  V_R = V_L &= \frac 1 2 (V_0^- + V_0^+) + \frac 1 2 \sqrt{A} (U_0^- - U_0^+).
\end{align}
\end{subequations}
The relation between trace and coupling data inferred by ~\eqref{eq:laxcondition} and~\eqref{eq:couplingconditionsimplified} therefore defines a RS with easy-to-compute coupling data. In general the coupling condition~\eqref{eq:kirchhoffcouplingcondition} does not lead to consistency with the original Kirchhoff coupling problem, in which~\eqref{eq:hyperbolicsystem11} is endowed with the coupling condition
\[
 \Psi_U^K[U(t, 0^-), U(t, 0^+)] = F_1(U(t, 0^-)) -F_2(U(t, 0^+)) = 0
\]
This is due to the lower $n$ components in~\eqref{eq:kirchhoffcouplingcondition} that in general impose additional conditions in the relaxation limit. Instead concluding from~\eqref{eq:linearcouplinginu} consistency is given in case of the coupling condition
\[
  \Psi_U[U(t, 0^-), U(t, 0^+)] =
  \begin{pmatrix}
    U(t, 0^-) -  U(t, 0^+) \\ F_1(U(t, 0^-)) -F_2(U(t, 0^+)).
  \end{pmatrix}
\]
\end{Ex}

\subsection{Discrete relaxation limit}\label{sec:discretelimit}
Similarly as for the continuous problem in Section~\ref{sec:continuouslimit}, we can analyze the relaxation limit of the RS. Therefore, we assume that the trace data $Q_0=(Q_0^-, Q_0^+)$ can be asymptotically expanded around the state $Q_0^0=\left( Q_0^{0,-}, Q_0^{0,+} \right)$ as
\begin{equation}
Q_0^- = Q_0^{0,-} + \sum_{k=1}^\infty \varepsilon^k Q_0^{k,-}, \qquad Q_0^+ = Q_0^{0,+}  + \sum_{k=1}^\infty \varepsilon^k Q_0^{k,+}.
\end{equation}

Moreover, we assume in analogy to Lemma~\ref{lem:implicit} that $\Psiq$ is continuously differentiable and that for all $Q_0$ there is a unique $(\Sigma^-(Q_0), \Sigma^+(Q_0))=\Sigma\in \R^{2n}$ satisfying $\Psiqp[\Sigma; Q_0]=0$ and the Jacobian $D_\Sigma\Psiqp[\Sigma; Q_0]$ is invertible. From the parametrization~\eqref{eq:couplingdataparameterized} we obtain
\begin{equation}\label{eq:limitpara}
  Q_R(\Sigma; Q_0) = Q_0^{-} + R_1^- \Sigma^-(Q_0), \qquad Q_L(\Sigma; Q_0) = Q_0^{+} + R_2^+ \Sigma^+(Q_0),
\end{equation}
Due to the implicit function theorem, see the proof of Lemma~\ref{lem:implicit}, $\Sigma^-$ and $\Sigma^+$ are continuously differentiable functions in $Q_0$ and by Taylor expansion
\begin{equation}
\Sigma^\mp(Q_0) = \Sigma^\mp(Q_0^0) + \mathcal{O}(\varepsilon)
\end{equation}
holds. Hence, we obtain by substitution in~\eqref{eq:limitpara}
\begin{equation}
  Q_R(\Sigma; Q_0) = Q_0^{0,-} + R_1^- \Sigma^-\left(Q_0^0\right) + \mathcal{O}(\varepsilon), \quad Q_L(\Sigma; Q_0) = Q_0^{0,+} + R_2^+ \Sigma^+\left(Q_0^0\right) + \mathcal{O}(\varepsilon)
\end{equation}
or, in different terms
\begin{equation}
  \mathcal{RS}(Q_0^-, Q_0^+) = \mathcal{RS}\left(Q_0^{0,-}, Q_0^{0,+}\right) + \mathcal{O}(\varepsilon).
\end{equation}
Using a Taylor expansion of $\Psiq$ we further see that
\begin{equation*}
  \Psiq[Q_R(\Sigma(Q_0); Q_0), Q_L(\Sigma(Q_0); Q_0)] = \Psiq[Q_R(\Sigma(Q_0^0); Q_0^0), Q_L(\Sigma(Q_0^0); Q_0^0)] + \mathcal{O}(\varepsilon).
\end{equation*}
and thus, in the limit $\varepsilon \rightarrow 0$ the coupling data satisfies the coupling condition, i.e.,
\begin{equation}
  \Psiq[Q_R(\Sigma(Q_0^0); Q_0^0), Q_L(\Sigma(Q_0^0); Q_0^0)] = 0.
\end{equation}

\section{The central scheme}\label{sec:schemes}
\begin{figure}
  \input{input/one-one-num-coupling.tikz}
  \caption{The relaxation system in the 1-to-1 coupling case on the discretized real line.}\label{fig:one-to-one-num}
\end{figure}
We consider a discretization of the real line into uniform mesh cells $I_j=(x_{j-1/2}, x_{j+1/2})$ of width $\Delta x$ with origin located at the cell interface $x_{-1/2}$ as indicated by Figure~\ref{fig:one-to-one-num}. Furthermore, the time line is partitioned into the instances $t^n=\sum_{k=1}^n \Delta t^n$ for some time increments $\Delta t^n>0$, which for brevity we assume uniform and denote by $\Delta t$ throughout this section. Let $Q_j^n$, $U_j^n$, and $V_j^n$ denote a piecewise constant numerical solution of system~\eqref{eq:relsyst11} in terms of cell averages over cell $I_j$ at time $t^n$.

\subsection{Relaxation scheme}
An asymptotic preserving scheme for the coupled relaxation system~\eqref{eq:relsyst11} and coupling condition~\eqref{eq:relaxationcoupling} is obtained by combining the first order upwind discretization with an implicit-explicit time discretization as in~\cite{herty2023centr, hu2017asymppreserschem}. Away from the interface the scheme reads
\begin{subequations}\label{eq:relscheme11}
  \begin{align}
        U^{n+1}_{j} &= U^{n}_{j} - \frac{\Delta t}{2 \Delta x}\left( V^{n}_{j+1} - V^{n}_{j-1}\right) + \frac{\Delta t}{2 \Delta x} \sqrt{A_i} \left( U^{n}_{j+1} - 2 U^{n}_{j} + U^{n}_{j-1}\right), \\
    V^{n+1}_{j} &= V^{n}_{j} - \frac{\Delta t}{2 \Delta x} A_i \left( U^{n}_{j+1} - U^{n}_{j-1}\right) + \frac{\Delta t}{2 \Delta x} \sqrt{A_i}\left( V^{ n}_{j+1} - 2 V^{n}_{j} + V^{n}_{j-1}\right) + \frac {\Delta t} \eps \left( F_i(U^{n+1}_{j}) - V^{n+1}_{j} \right),\label{eq:unsplitv0}
  \end{align}
for $j\in \Z \setminus \{0, -1\}$, where $i=1$ if $j<0$ and $i=2$ if $j>0$. For the cell that is left next to the coupling interface we have
  \begin{align}
        U^{n+1}_{-1} &= U^{n}_{-1} - \frac{\Delta t}{2 \Delta x}\left( V^{n}_{R} - V^{n}_{-2}\right) + \frac{\Delta t}{2 \Delta x} \sqrt{A_1} \left( U^{n}_{R} - 2 U^{n}_{-1} + U^{n}_{-2}\right), \\
    V^{n+1}_{-1} &= V^{n}_{-1} - \frac{\Delta t}{2 \Delta x} A_1 \left( U^{n}_{R} - U^{n}_{-2}\right) + \frac{\Delta t}{2 \Delta x} \sqrt{A_1}\left( V^{ n}_{R} - 2 V^{n}_{-1} + V^{n}_{-2}\right) + \frac {\Delta t} \eps \left( F_1(U^{n+1}_{-1}) - V^{n+1}_{-1} \right),\label{eq:unsplitvl}
  \end{align}
  and on the opposite site, for the cell right next to the coupling interface the scheme reads
    \begin{align}
        U^{n+1}_{0} &= U^{n}_{0} - \frac{\Delta t}{2 \Delta x}\left( V^{n}_{1} - V^{n}_{L}\right) + \frac{\Delta t}{2 \Delta x} \sqrt{A_2} \left( U^{n}_{1} - 2 U^{n}_{0} + U^{n}_{L}\right), \\
    V^{n+1}_{0} &= V^{n}_{0} - \frac{\Delta t}{2 \Delta x} A_2 \left( U^{n}_{1} - U^{n}_{L}\right) + \frac{\Delta t}{2 \Delta x} \sqrt{A_2}\left( V^{ n}_{1} - 2 V^{n}_{0} + V^{n}_{L}\right) + \frac {\Delta t} \eps \left( F_2(U^{n+1}_{0}) - V^{n+1}_{0} \right).\label{eq:unsplitv2}
    \end{align}
\end{subequations}
The discrete coupling data employed in scheme~\eqref{eq:relscheme11} is obtained applying the RS from Section~\ref{sec:RS} to the cell averages adjacent to the interface, i.e.,
\begin{equation}
  \mathcal{RS}( (U_{-1}^n, V_{-1}^n), (U_0^n, V_0^n)) := ((U_R^n, V_R^n), (U_L^n, V_L^n)).
\end{equation}

Scheme~\eqref{eq:relscheme11} can be rewritten in the form
\begin{subequations}
\begin{align}
  U^{n+1}_{j} &= U^{n}_{j} - \frac{\Delta t}{\Delta x}\left( F_{j+1/2}^{n,-} - F_{j-1/2}^{n,+} \right) \label{eq:relaxationschemeu},\\
  V^{n+1}_{j} &= V^{n}_{j} - \frac{\Delta t}{\Delta x}\left( G_{j+1/2}^{n,-} - G_{j-1/2}^{n,+} \right) + \frac{\Delta t}{\varepsilon} \left( F_i(U^{n+1}_j) - V^{n+1}_j)\right), \label{eq:relaxationschemev}
\end{align}
\end{subequations}
where $i=1$ if $j<0$ and $i=2$ if $j\geq 0$. We distinguish between incoming numerical fluxes from the left interface (indicated by a plus sign) and numerical fluxes outgoing through the right interface (indicated by a minus sign). Away from the interface they are equal as we have
\begin{subequations}
\begin{align}
  F_{j-1/2}^{n,+} = F_{j-1/2}^{n,-} =  \frac 1 2 \, (V_{j-1}^n + V_{j}^n)  - \frac 1 2  \sqrt{A_i} (U_j^n-U_{j-1}^n),\\
  G_{j-1/2}^{n,+} = G_{j-1/2}^{n,-} =  \frac 1 2 A_i  (U_{j-1}^n + U_{j}^n)  - \frac 1 2  \sqrt{A_i} (V_j^n-V_{j-1}^n) \label{eq:vflux}           
\end{align}
for $j\in \Z\setminus\{ 0\}$ and $i$ chosen according to the sign of $j$ as in~\eqref{eq:relscheme11}. At the interface the scheme is not necessarily conservative since the numerical fluxes are given by
\begin{align}
  F_{-1/2}^{n,-} &=  \frac 1 2 \, (V_{-1}^n + V_{R}^n)  - \frac 1 2  \sqrt{A_1} (U_R^n-U_{-1}^n),\\
  F_{-1/2}^{n,+} &=  \frac 1 2 \, (V_{L}^n + V_{0}^n)  - \frac 1 2  \sqrt{A_2} (U_0^n-U_{L}^n),\\
  G_{-1/2}^{n,-} &=  \frac 1 2 A_1  (U_{-1}^n + U_{R}^n)  - \frac 1 2  \sqrt{A_1} (V_R^n-V_{-1}^n),\\
  G_{-1/2}^{n,+} &=  \frac 1 2 A_2  (U_{L}^n + U_{0}^n)  - \frac 1 2  \sqrt{A_2} (V_0^n-V_{L}^n).            
\end{align}          
\end{subequations}

\subsection{Relaxed scheme}\label{sec:relaxedscheme}
In this section, using asymptotic analysis as done in Section~\ref{sec:continuouslimit} and~\ref{sec:discretelimit} we consider the relaxation limit of the relaxation scheme. The cell averages $U_j^n$ and $V_j^n$ occuring in the schemes above depend on the relaxation rate $\varepsilon$. We assume that for all $j$ and $n$ they can be asymptotically expanded around the zero-relaxation states $U_j^{n,0}$ and $V_j^{n,0}$ for sufficiently small relaxation rates as
\begin{subequations}\label{eq:expansion}
\begin{align}
  U_j^{n} = U_j^{n,0} + \varepsilon U_j^{n,1} + \mathcal{O}(\varepsilon^2),\\
  V_j^{n} = V_j^{n,0} + \varepsilon V_j^{n,1} + \mathcal{O}(\varepsilon^2).
\end{align}
\end{subequations}
At first we consider the relaxation limit away from the interface. We fix $j\in \Z\setminus \{ -1, 0\}$ and choose $i$ according to the sign of $j$ as in~\eqref{eq:relscheme11} and then substitute the expansions~\eqref{eq:expansion} into~\eqref{eq:relaxationschemev}. After Taylor expansion we obtain
\begin{equation}\label{eq:vexpansion}
  \begin{split}
    V^{n+1, 0}_{j} \left(1 + \frac{\Delta t}{\varepsilon} \right) &= V^{n, 0}_{j} - \frac{\Delta t}{\Delta x}\left( G_{j+1/2}^{n,+, 0} - G_{j-1/2}^{n,-, 0} \right) \\ &\quad + \frac{\Delta t}{\varepsilon} \left( F_i(U^{n+1, 0}_j)  + \varepsilon DF_i(U^{n+1, 0}_j)~U_j^{n+1,1} - \varepsilon V_j^{n,1} \right) + \mathcal{O}(\varepsilon),
    \end{split}
\end{equation}
where $DF_i$  denotes the Jacobian of $F_i$ and  $G_{j-1/2}^{n,\mp, 0}$= $G_{j-1/2}^{n,\mp, 0}\left((U^{n,0}_{j-1}, U^{n,0}_{j-1}), (U^{n,0}_{j}, U^{n,0}_{j})  \right)$ the numerical flux equivalents to~\eqref{eq:vflux} at the zero-relaxation state. Due to the expansion
\begin{equation}
 \left( 1 + \frac{\Delta t}{\varepsilon} \right)^{-1} = \frac{\varepsilon}{\varepsilon + \Delta t} = \frac{\varepsilon}{\Delta t} + \mathcal{O}(\varepsilon^2) 
\end{equation}
we obtain by multiplication in~\eqref{eq:vexpansion}
\begin{equation}\label{eq:vlimitpre}
   V^{n+1, 0}_{j} = F_i(U^{n+1, 0}_j) + \mathcal{O}(\varepsilon) \qquad \text{for all }n\in \N_0.
 \end{equation}
 % As the zero-relaxation states do not depend on $\varepsilon$ any more, this implies
 % \begin{equation}\label{eq:vlimit}
 %   V^{n+1, 0}_{j} = F_i(U^{n+1, 0}_j) \qquad \text{for all }n\in \N_0.
 % \end{equation}
 Substituting now the asymptotic expansion~\eqref{eq:expansion} into~\eqref{eq:relaxationschemeu} and taking into account~\eqref{eq:vlimitpre} we obtain
 \begin{equation}\label{eq:asymptoticlimitscheme}
   U^{n+1,0}_{j} = U^{n,0}_{j} - \frac{\Delta t}{\Delta x}\left( F_{j+1/2}^{n,+,0} - F_{j-1/2}^{n,-,0} \right)  + \mathcal{O}(\varepsilon),
 \end{equation}
 where the numerical fluxes take the form
 \begin{equation}\label{eq:asymptoticfluxes}
     F_{j-1/2}^{n,+,0} = F_{j-1/2}^{n,-,0} =  \frac 1 2 \, (F_i(U_{j-1}^{n,0}) + F_i(U_{j}^{n,0}))  - \frac 1 2  \sqrt{A_i} (U_j^{n,0}-U_{j-1}^{n,0}).
 \end{equation}
In the relaxation limit the $\mathcal{O}(\varepsilon)$ terms in~\eqref{eq:vlimitpre} and ~\eqref{eq:asymptoticlimitscheme} vanish, which gives rise to the relaxed scheme on both half-axes.  

 Next, we investigate the relaxation limit at the interface. We assume that in case of the asymptotic expansion~\eqref{eq:expansion} the RS~\eqref{eq:RSdef} satisfies
 \begin{equation}\label{eq:RSlimitassumption}
  \mathcal{RS}( (U_{-1}^n, V_{-1}^n), (U_0^n, V_0^n)) =  \mathcal{RS}_{0}( (U_{-1}^{n,0}, V_{-1}^{n,0}), (U_0^{n,0}, V_0^{n,0})) + \mathcal{O}(\varepsilon)
\end{equation}
for some mapping
\[
  \mathcal{RS}_{0}: \R^{2n} \times \R^{2n} \rightarrow \R^{2n} \times \R^{2n}.
\]
In Section~\ref{sec:discretelimit} we have shown that under suitable conditions on $\Psi_Q$ condition~\eqref{eq:RSlimitassumption} holds for $\mathcal{RS}_0=\mathcal{RS}$, i.e., the coupling in the limit is determined by the original RS. We note however, that in general $\mathcal{RS}_{0}$ may not necessarily be a RS in the sense it was introduced in Section~\ref{sec:halfRiemann}.

 Condition~\eqref{eq:RSlimitassumption} allows us to define the limit coupling data
  \begin{equation}\label{eq:RSasymptotic}
  \mathcal{RS}_{0}( (U_{-1}^{n,0}, V_{-1}^{n,0}), (U_0^{n,0}, V_0^{n,0})) =: ((U_R^{n,0}, V_R^{n,0}), (U_L^{n,0}, V_L^{n,0}))
\end{equation}
and therefore also the interface fluxes
\begin{subequations}
\begin{align}
  F_{-1/2}^{n,-,0} &=  \frac 1 2 \, (V_{-1}^{n,0} + V_{R}^{n,0})  - \frac 1 2  \sqrt{A_1} (U_R^{n,0}-U_{-1}^{n,0}),\\
  F_{-1/2}^{n,+,0} &=  \frac 1 2 \, (V_{L}^{n,0} + V_{0}^{n,0})  - \frac 1 2  \sqrt{A_2} (U_0^{n,0}-U_{L}^{n,0}),\\
  G_{-1/2}^{n,-, 0} &=  \frac 1 2 A_1  (U_{-1}^{n,0} + U_{R}^{n,0})  - \frac 1 2  \sqrt{A_1} (V_R^{n,0}-V_{-1}^{n,0}),\\
  G_{-1/2}^{n,+, 0} &=  \frac 1 2 A_2  (U_{L}^{n,0} + U_{0}^{n,0})  - \frac 1 2  \sqrt{A_2} (V_0^{n,0}-V_{L}^{n,0}).            
\end{align}          
\end{subequations}
Combining these fluxes with the above arguments and using assumption~\eqref{eq:RSlimitassumption} it is now straightforward to show that~\eqref{eq:vlimitpre} and~\eqref{eq:asymptoticlimitscheme} also hold for $j=-1$ and $j=0$.  Eventually we take the relaxation limit $\eps \rightarrow 0$, drop the index $0$ and obtain the \emph{coupled relaxed} or \emph{central scheme} for system~\eqref{eq:relsyst11} and coupling condition~\eqref{eq:relaxationcoupling} in the relaxation limit, which reads
\begin{subequations}\label{eq:relaxedscheme}
\begin{equation}\label{eq:conservativeform}
  U_j^{n+1} = U_j^n - \frac{\Delta t }{\Delta x} \left( H_{j+1/2}^{n,-} -  H_{j-1/2}^{n,+}\right) \quad \text{for all }j \in \Z
\end{equation}
and employs the numerical fluxes
\begin{align}
   H_{j-1/2}^{n,+} = H_{j-1/2}^{n,-} =  \frac 1 2 \, (F_i(U_{j-1}^n) + F_i(U_{j}^n))  - \frac 1 2  \sqrt{A_i} (U_j^n-U_{j-1}^n)
\end{align}
for $j\in \Z\setminus \{0\}$ and $i$ chosen according to the sign of $j$. At the interface we have
\begin{align}
  H_{-1/2}^{n,-} &=  \frac 1 2 \, (F_1(U_{-1}^n) + V_{R}^n)  - \frac 1 2  \sqrt{A_1} (U_R^n-U_{-1}^n)\\
  H_{-1/2}^{n,+} &=  \frac 1 2 \, (V_{L}^n + F_2(U_{0}^n))  - \frac 1 2  \sqrt{A_2} (U_0^n-U_{L}^n),
\end{align}
where the variables $V_R^n=V_R^n(U_{-1}^n, U_{0}^n)$, $V_L^n=V_L^n(U_{-1}^n, U_{0}^n)$ depend only on the cell averages of the state variable next to the interface as part of the coupling data
  \begin{equation}\label{eq:RSlimit}
      \mathcal{RS}_{0}( (U_{-1}^n, F_1(U_{-1}^n), (U_0^n, F_2(U_0^n))) =: ((U_R^n, V_R^n), (U_L^n, V_L^n)).
    \end{equation}
  \end{subequations}
  
  \begin{Rmk}\label{rmk:limitrs}
   Suppose that $\Psi_Q$ is continuously differentiable, for all $Q_0$ there is a unique $\Sigma\in\R^n$ satisfying $\Psiqp[\Sigma; Q_0]=0$ and the Jacobian $D_\Sigma\Psiqp[\Sigma; Q_0]$ is invertible. Then following from Section~\ref{sec:discretelimit} the RS for the coupling condition~\eqref{eq:psigeneral} is well defined and $\mathcal{RS}_{0} \coloneqq \mathcal{RS}$ satisfies~\eqref{eq:RSlimitassumption} and defines a limit scheme.
  \end{Rmk}  

  \begin{Ex}\label{ex:kirchhoffscheme}
    If Kirchhoff conditions are assumed for $\Psi_Q$ together with $A_1=A_2=A$ as discussed in Example~\ref{ex:KirchhoffCoupling} we can substitute the explicit formulas of the coupling data and obtain a conservative scheme of the form~\eqref{eq:conservativeform} with the numerical fluxes
    \begin{equation}
      H_{j-1/2}^{n,+}=  H_{j-1/2}^{n,-} = \begin{cases}
                                            \frac 1 2 \, (F_1(U_{j-1}^n) + F_1(U_{j}^n))  - \frac 1 2  \sqrt{A} (U_j^n-U_{j-1}^n)  &\text{if }j<0, \\[5pt]
                                            \frac 1 2 \, ( F_1(U_{-1}^n) + F_2(U_{0}^n))  - \frac 1 2 \sqrt{A} (U_0^n-U_{-1}^n) &\text{if }j=0, \\[5pt]
                                            \frac 1 2 \, ( F_2(U_{j-1}^n) + F_2(U_{j}^n))  - \frac 1 2 \sqrt{A} (U_j^n-U_{j-1}^n) &\text{if }j>0. 
                                          \end{cases}
                                        \end{equation}
This follows from generalizing the scalar approach in~\cite{herty2023centr} to the system case.
\end{Ex}

\section{Case study: the $p$-system}\label{sec:psystem}
In this section we apply the developed approach to a problem arising in gas transportation networks considered in~\cite{herty2019couplcompreulerequat}. Motivated by recent significant increases of gas-fired electrical generation in parts of the USA, see e.g.~\cite{zlotnik2016contr}, the modeling of high-pressure gas transportation networks needs to account for the influence of gas-fired power generators. The gas flow through pipelines has been effectively modeled using the compressible Euler equations~\cite{brouwer2011gas}. Neglecting temperature dynamics the $p$-system is used as a simplified model of gas dynamics, see e.g.,~\cite{herty2010}, given by
\begin{subequations}\label{eq:psystem}
  \begin{align}
    \ddt \rho + \ddx(\rho v) &= 0, \\
    \ddt (\rho v) + \ddx(\rho v^2 + p(\rho)) &=0
  \end{align}
and accounting for the conservation of density $\rho$ and momentum $\rho v$.  As closure for the system we impose the $\gamma$-law for the pressure function
  \begin{equation}
p(\rho) = \alpha \rho^{\gamma}
\end{equation}
\end{subequations}
with parameters $\alpha>0$ and $\gamma\geq0$. To account for the gas-fired power generation we consider~\eqref{eq:psystem} on both real half-axes coupled at an interface located at $x=0$ in analogy to~\eqref{eq:hyperbolicsystem11}. In this setting a gas turbine acting as the power generator is assumed at the interface. The impact of the turbine on the gas network is modeled by conditions that leave the pressure of the transmitted gas constant but account for a loss in momentum~\cite{chertkov2015pressfluctnatur} giving rise to the coupling conditions
\begin{subequations}\label{eq:psystemjumpcondition}
\begin{align} 
  \rho_R v_R &= \rho_L v_L + \mathcal{E}, \label{eq:rhovjump}\\
 p(\rho_R) &=  p(\rho_L) \label{eq:psystemjumpconditionpressure}
\end{align}
\end{subequations}
 which govern the coupling function $\Psi_U$. In condition~\eqref{eq:rhovjump} the parameter $\mathcal{E}\leq0$ controls the momentum outtake due to the turbine. As the density $\rho$ is assumed nonnegative we note that assuming $\gamma>0$ condition~\eqref{eq:psystemjumpconditionpressure} imposing continuity of the pressure at the interface is equivalent to
\begin{equation}\label{eq:psystemcontinuousdensity}
\rho_R = \rho_L.
\end{equation}

\subsection{Relaxation form}\label{sec:psystemrelaxation}
%In this section we state the relaxation form of the $p$-system~\eqref{eq:psystem} and discuss suitable coupling conditions.
% In this application we consider equal flux functions left and right from the interface.% and allow for different parameters $\alpha_1$, $\gamma_1$ left from the interface and $\alpha_2$, $\gamma_2$ right from the interface.
Introducing the two auxiliary scalar variables $V_1$, $V_2$ and the relaxation rate $\varepsilon$ the relaxation system corresponding to the coupled $p$-system~\eqref{eq:psystem} takes the form
\begin{subequations}\label{eq:psystemrelaxation}
  \begin{align}
    \ddt \rho + \ddx(V_1) &= 0, \\
    \ddt (\rho v) + \ddx(V_2) &=0, \\
    \ddt V_1 + a \ddx \rho &= \frac 1 \varepsilon (\rho v - V_1), \\
    \ddt V_2 + a \ddx (\rho v)  &= \frac 1 \varepsilon (\rho v^2 + p(\rho) - V_2).
  \end{align}
\end{subequations}
As system~\eqref{eq:psystem} also its relaxation variant~\eqref{eq:psystemrelaxation} is imposed on both real half-axes coupled at an interface located at $x=0$. Compared to the general relaxation form~\eqref{eq:relsyst11} the parameters $a_i^j$ are chosen independent of both the half-axis and the system component and set to the square of the maximal speed of sound, i.e.,
\[
  a = \max_{\rho \in[0, \rho^\text{max}]} p^\prime(\rho).
  \]
In vector form system~\eqref{eq:psystemrelaxation} uses the variables $U=(\rho, \rho v)$ and $V=(V_1, V_2)$ and the matrix $A=aI$.
% \subsection{Coupling conditions}
% % We first assume equal flux functions on both sides of the interface and hence assume $p \coloneq p_1 \equiv p_2$. In addition
% In this section we discuss coupling condition for the $p$-system~\eqref{eq:psystem} and its relaxation counterpart~\eqref{eq:psystemrelaxation}.
% Concerning the former, 

We are interested in identifying a suitable coupling function $\Psi_Q$ for the coupled relaxation system~\eqref{eq:psystemrelaxation}. The implied condition should be such that in the limit $\varepsilon \rightarrow 0$ the dynamics of the original coupled $p$-system given by~\eqref{eq:psystem} and~\eqref{eq:psystemjumpcondition} are attained. At first we look at linear coupling conditions of the type 
\begin{subequations}\label{eq:psystemrelaxationcoupling}
  \begin{align}
    \rho_R &= \rho_L, \label{eq:relaxationrhocondition}\\
    \rho_R v_R &= \rho_L v_L + \beta_1 \mathcal{E}, \label{eq:rhovjumpagain}\\ 
     V_{1R} &= V_{1L} + \beta_2 \mathcal{E}, \label{eq:v1condition}\\
    V_{2R} &= V_{2L} \label{eq:v2continuity}
  \end{align}
\end{subequations}
for $\beta_1$, $\beta_2\in\{0,1\}$. As first approach motivated from~\eqref{eq:rhovjump} and~\eqref{eq:psystemcontinuousdensity} we consider the case $\beta_1=1$ and $\beta_2=0$.
The conditions~\eqref{eq:v1condition} (with $\beta_2=0$) and~\eqref{eq:v2continuity} for the auxiliary variable $V$ have been chosen assuming continuity of the auxiliary variable.
Noting that in the relaxation limit it holds  $V_{1}= \rho v$, the jump condition~\eqref{eq:rhovjump} can also be implemented in~\eqref{eq:v1condition}. This motivates our second approach, for which we choose $\beta_1=0$ and $\beta_2=1$. Our third approach combines the two former ones by setting $\beta_1=\beta_2=1$.

By our analysis in Section~\ref{sec:continuouslimit} and the linearity of the coupling approaches 1--3 the auxiliary variables satisfy in the relaxation limit
\begin{equation}
   V_{1R} = \rho_R v_R, \quad V_{1L} = \rho_L v_L, \quad V_{2R} = \rho_R (v_R)^2 + p(\rho_R), \quad V_{2L} = \rho_L (v_L)^2 + p(\rho_L).
\end{equation}
This limit condition shows that both our first and our second coupling approach for system~\eqref{eq:psystemrelaxation} lead to a contradiction in~\eqref{eq:rhovjumpagain} and~\eqref{eq:v1condition} in the relaxation limit due to $\beta_1\neq \beta_2$. Our third approach corrects this, however due to~\eqref{eq:v2continuity} this approach as well as the two former ones impose a condition in the limit that is not implied by the original coupling condition~\eqref{eq:psystemjumpcondition}. In more details, this additional condition imposes $v_R=v_L$, whenever $\rho_R v_R= \rho_Lv_L \neq 0$. Consequently, our three coupling approaches within~\eqref{eq:psystemrelaxationcoupling} cannot lead to consistency of the relaxation system~\eqref{eq:psystemrelaxationcoupling} with the p-system~\eqref{eq:psystem} and coupling condition~\eqref{eq:psystemjumpcondition} in the sense of Definition~\ref{def:consistency}.

We aim to design a coupling condition that leads to consistency by taking our third approach as a starting point and deriving a suitable replacement for condition~\eqref{eq:v2continuity}. Therefore we assume~\eqref{eq:rhovjump},~\eqref{eq:psystemcontinuousdensity} as well as $\rho_R=\rho_L\neq 0$ and find
\begin{equation}
\rho_R v_R^2 + p(\rho_R) = \frac{\rho_R^2 v_R^2}{\rho_R} + p(\rho_R) = \frac{(\rho_L v_L+ \mathcal{E})^2 }{\rho_R} + p(\rho_L) = \rho_L v_L^2 + p(\rho_L) + \mathcal{E}\, \frac{ 2 \rho_L v_L + \mathcal{E}}{\rho_L}.
\end{equation}
This motivates the new condition
\begin{equation}\label{eq:nonlinearV2condition}
V_{2R} = V_{2L} + \mathcal{E}\, \frac{ 2 \rho_L v_L + \mathcal{E}}{\rho_L},
\end{equation}
which together with~\eqref{eq:relaxationrhocondition},~\eqref{eq:rhovjumpagain},~\eqref{eq:v1condition} and $\beta_1=\beta_2=1$ determines our fourth coupling approach. By construction the relaxation system~\eqref{eq:psystemrelaxationcoupling} under this condition is consistent with the coupled $p$-system given by~\eqref{eq:psystem} and~\eqref{eq:psystemjumpcondition}. Unlike the previous approaches this condition is nonlinear.

\subsection{Riemann solvers}
In this section we analyze the Riemann solvers for the different coupling conditions considered in Section~\ref{sec:psystemrelaxation} and provide explicit forms for the implied coupling data,

\paragraph{Coupling approaches 1--3} The first three coupling conditions given by~\eqref{eq:psystemrelaxationcoupling} and different choices for $\beta_1$ and $\beta_2$ fit into the linear framework discussed in Section~\ref{sec:linearcoupling} and correspond to
\begin{equation}
  B_R = I, \qquad B_L = I, \qquad P =
  \begin{pmatrix}
    0 & \beta_1 \, \mathcal{E} & \beta_2 \, \mathcal{E} & 0
  \end{pmatrix}^T
\end{equation}
in the coupling function~\eqref{eq:linearcoupling}. To compute the coupling data from the RS we conclude from Remark~\ref{rmk:linearblockmatrices} that the system matrix in~\eqref{eq:sigmaexplicit} is of the form
\begin{equation}
  B \coloneq B_R \tilde R^- - B_L \tilde R^+ =
 - \begin{pmatrix}
    ( \sqrt{A})^{-1} & ( \sqrt{A})^{-1} \\
    -I & I
   \end{pmatrix}
   =
   -\begin{pmatrix}
     \frac 1 2 \sqrt{A} & -\frac 1 2 I \\
     \frac 1 2 \sqrt{A} & \frac 1 2 I
   \end{pmatrix}^{-1}
   ,
\end{equation}
from which we determine the parameter $\Sigma^-$ and $\Sigma^+$ as
\begin{equation}\label{eq:psystemrelaxationcouplingsigma}
  \begin{pmatrix}
\Sigma^- \\ \Sigma^+
  \end{pmatrix}
  = B^{-1} (P + Q_0^+ - Q_0^{-}), \qquad
  \Sigma^\mp = - \frac 1 2 \sqrt{A}(P^U + U_0^+ - U_0^-) \pm \frac 1 2 (P^V + V_0^+ - V_0^-).
\end{equation}
Here $P^U=(0, \beta_1 \mathcal{E})$ and $P^V=(\beta_2 \mathcal{E}, 0)$ consist of the upper two and lower two entries in $P$, respectively. The trace data has the form
\begin{equation}\label{eq:psystemrelaxationtracedata}
  Q_0^\mp =
  \begin{pmatrix}
U_0^\mp \\ V_0^\mp
  \end{pmatrix}
  =
  \begin{pmatrix}
  \rho_0^\mp & \rho_0^\mp v_0^\mp & V_{1,0}^\mp & V_{2,0}^\mp
  \end{pmatrix}^T.
\end{equation}
Eventually, the coupling data, for which we assume a notation analogue to~\eqref{eq:psystemrelaxationtracedata}, is given by~\eqref{eq:affinelinearrsoutput} implying
\begin{subequations}\label{eq:psystemrelaxationcouplingdata}
\begin{align}
  U_R &= \frac 1 2 (U_0^- + U_0^+ + P^U) - \frac 1 2 (\sqrt{A})^{-1}(V_0^+ - V_0^- + P^V), \\ 
  U_L &= \frac 1 2 (U_0^- + U_0^+ - P^U) - \frac 1 2 (\sqrt{A})^{-1}(V_0^+ - V_0^- + P^V), \\
  V_R &= \frac 1 2 \sqrt{A} (U_0^- - U_0^+ - P^U) + \frac 1 2 (V_0^+ + V_0^- + P^V),\\
  V_L &= \frac 1 2 \sqrt{A} (U_0^- - U_0^+ - P^U) + \frac 1 2 (V_0^+ + V_0^- - P^V).
\end{align}
\end{subequations}
Due to Remark~\ref{rmk:limitrs}, which applies here since $\Psi_Q$ is affine linear, the coupling data in the relaxation limit is determined by the same RS as in case of $\varepsilon>0$. Thus the data~\eqref{eq:psystemrelaxationcouplingdata} also satisfies the coupling condition~\eqref{eq:psystemrelaxationcoupling} in the limit $\varepsilon \rightarrow 0$.

\paragraph{Coupling approach 4} We next consider our fourth coupling approach, which introducing the notation $g^V(U_L) = \mathcal{E} (0, (2 \rho_L v_L + \mathcal{E})/\rho_L)^T$ can be written as
\begin{equation}
  \begin{pmatrix}
    U_R \\ V_R
  \end{pmatrix}
  =
  \begin{pmatrix}
    U_L \\ V_L
  \end{pmatrix}
  +
    \begin{pmatrix}
    P^U \\ P^V
    \end{pmatrix}
    +
      \begin{pmatrix}
    0 \\ g^V(U_L)
  \end{pmatrix}.
\end{equation}
To obtain coupling data we substitute the parametrization~\eqref{eq:couplingdataparameterized} and obtain
\begin{equation}\label{eq:solve4parametrization}
  \begin{pmatrix}
   - (\sqrt{A})^{-1} \\ I
  \end{pmatrix}
  \Sigma^- -
  \begin{pmatrix}
    (\sqrt{A})^{-1} \\ I
  \end{pmatrix}
  \Sigma^+ -
      \begin{pmatrix}
    0 \\ g^V(U_0^+ + (\sqrt{A})^{-1} \Sigma^+)
      \end{pmatrix}
      =
    \begin{pmatrix}
    U_0^+ - U_0^- + P^U \\ V_0^+ - V_0^- + P^V
    \end{pmatrix}.
\end{equation}
From the first two components of~\eqref{eq:solve4parametrization} we get
\begin{equation}\label{eq:solve4Ucomponents}
\Sigma^- =   \sqrt{A}(U_0^- -U_0^+ - P^U) -\Sigma^+,
\end{equation}
which substituted into the components three and four in~\eqref{eq:solve4parametrization} yields
\begin{equation}\label{eq:solve4Vcomponents}
- 2 \Sigma^+ + \sqrt{A}(U_0^- -U_0^+ - P^U) - g^V(U_0^+ + (\sqrt{A})^{-1} \Sigma^+) =  V_0^+ - V_0^- + P^V.
\end{equation}
Using the notation $\Sigma^+=(\sigma_1^+, \sigma_2^+)^T$ we note that the first component of~\eqref{eq:solve4Vcomponents} determines $\sigma_1^+$, which in turn substituted into the second component of~\eqref{eq:solve4Vcomponents} determines $\sigma_2^+$. Eventually $\Sigma^-$ is obtained from~\eqref{eq:solve4Ucomponents} and therefore a RS for approach four is defined. By the verified well-posedness of the RS and the differentiability of the corresponding coupling function Remark~\ref{rmk:limitrs} also applies to our fourth coupling approach and thus the above procedure also determines suitable coupling data in the relaxation limit.

\subsection{Numerical experiments}
We consider a numerical test case employing constant initial data $\rho_0=\rho_0 v_0=1$ and pressure given by $p(\rho)= 146\,820.4 \rho$, see Section 5.2.1 in~\cite{sikstel2020analy}. As computational domain we choose $[-200,200]$ and impose homogeneous Neumann boundary conditions. The momentum outtake at the interface is assumed time dependent and given by
\begin{equation}
  -\mathcal{E}(t) =
  \begin{cases}
    \min\{ 0.6t, 3t\} & \text{if }0\leq t < 0.
    \\ \max \{ 0, -3t + 1.5\} & \text{if } t \geq 0.3.
  \end{cases}
\end{equation}
A reference solution for the coupled $p$-system~\eqref{eq:psystem} and coupling condition~\eqref{eq:psystemjumpcondition} is presented in~\cite{sikstel2020analy}. 

We employ the central scheme introduced in Section~\ref{sec:relaxedscheme} to numerically compute solutions for system~\eqref{eq:psystemrelaxation} in the limit $\varepsilon \rightarrow 0$ for the different coupling approaches considered in Section~\ref{sec:psystemrelaxation}. In the numerical computations we employ 1000 uniform mesh cells and time increments chosen according to the CFL condition
\begin{equation}
  \Delta t = \text{CFL} \frac{\Delta x}{\sqrt{a}}
\end{equation}
and Courant number $\text{CFL}=0.49$. Numerical results are presented in terms of density, momentum and pressure at the time instances $t=0.0716$, $0.2864$ and $0.55$.

\begin{figure}
  \begin{tikzpicture}
  \begin{groupplot}[
        group style={group size=3 by 3,
            horizontal sep = .8 cm, 
            vertical sep = .6 cm,
            xlabels at=edge bottom,
            xticklabels at=edge bottom,
            yticklabels at=edge right}, 
          width = .355 \linewidth,
          height = .22 \linewidth,
          xlabel={$x$},
          xmin=-200,xmax=200,
          every tick label/.append style={font=\scriptsize},
          title style={font=\scriptsize},
          label style={font=\scriptsize},
          ]
        \nextgroupplot[title={$t=0.0716$}, ylabel={$\rho$}, ymin=0.9995, ymax=1.0001]
        \addplot [] table [x index=0, y index=1] {input/psystem11jumpCPL1_2.dat};
        \nextgroupplot[title={$t=0.2864$}, ymin=0.9995, ymax=1.0001]
        \addplot [] table [x index=0, y index=1] {input/psystem11jumpCPL1_3.dat};
        \nextgroupplot[title={$t=0.55$},
        scaled y ticks=manual:{$\cdot 10^{-4} + 1$}{
          \pgfmathparse{10000*(#1-1)}
        },
        ymin=0.9995, ymax=1.0001]
        \addplot [] table [x index=0, y index=1] {input/psystem11jumpCPL1_4.dat};
        \nextgroupplot[ylabel={$\rho v$}, ymin=0.8, ymax=1.2]
        \addplot [] table [x index=0, y index=2] {input/psystem11jumpCPL1_2.dat};
        \nextgroupplot[ymin=0.8, ymax=1.2]
        \addplot [] table [x index=0, y index=2] {input/psystem11jumpCPL1_3.dat};
        \nextgroupplot[ymin=0.8, ymax=1.2]
        \addplot [] table [x index=0, y index=2] {input/psystem11jumpCPL1_4.dat};
        \nextgroupplot[ylabel={$p$}, ymin=146760, ymax=146830]
        \addplot [] table [x index=0, y index=3] {input/psystem11jumpCPL1_2.dat};
        \nextgroupplot[ymin=146760, ymax=146830]
        \addplot [] table [x index=0, y index=3] {input/psystem11jumpCPL1_3.dat};
        \nextgroupplot[ymin=146760, ymax=146830,
        scaled y ticks=manual:{$+ 146\,820$}{
          \pgfmathparse{(#1-146820)}
        }]
        \addplot [] table [x index=0, y index=3] {input/psystem11jumpCPL1_4.dat};
  \end{groupplot}
\end{tikzpicture}
\caption{Numerical solution for system~\eqref{eq:psystemrelaxation} in the relaxation limit using coupling approach~1 given by~\eqref{eq:psystemrelaxationcoupling}, $\beta_1=1$ and $\beta_2=0$. Peaks appears at the interface ($x=0$). The coupled relaxed scheme on $1000$ mesh cells with $\text{CFL}=0.49$ has been used. }\label{fig:psystemcpl1}
\end{figure}
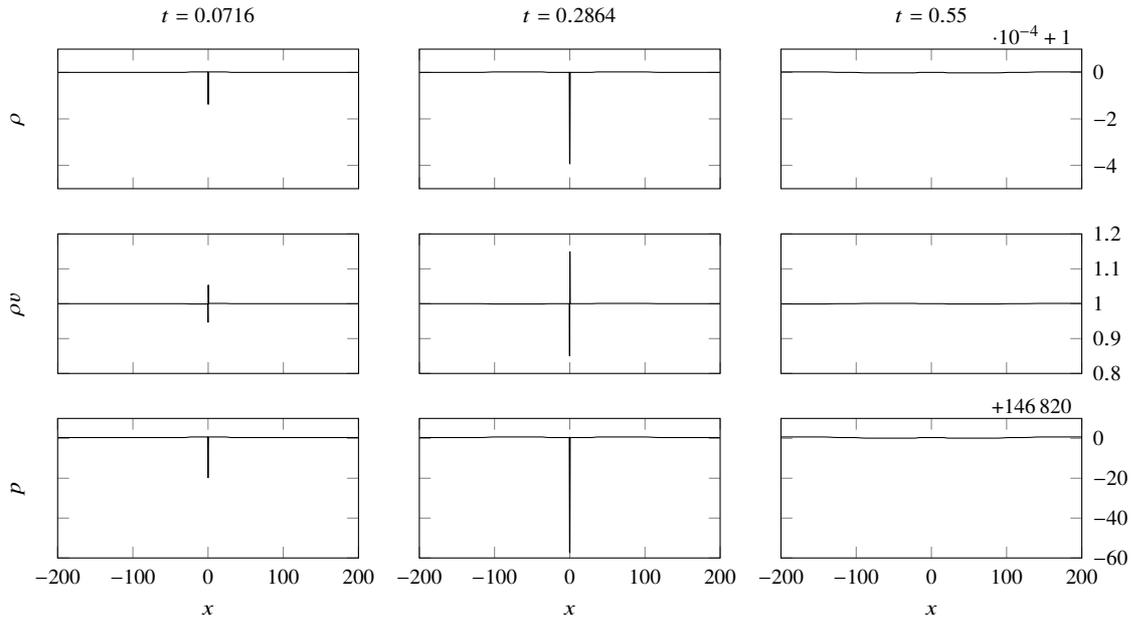

In Figure~\ref{fig:psystemcpl1} we show results obtained by the first coupling approach given by~\eqref{eq:psystemrelaxationcoupling}, $\beta_1=1$ and $\beta_2=0$. In the numerical results we see a peak in density, momentum and pressure forming and vanishing at the interface as time evolves. Comparing to the reference solution this coupling condition imposed to the coupled relaxation system clearly cannot yield the solution of the original coupled $p$-system~\eqref{eq:psystem}. 
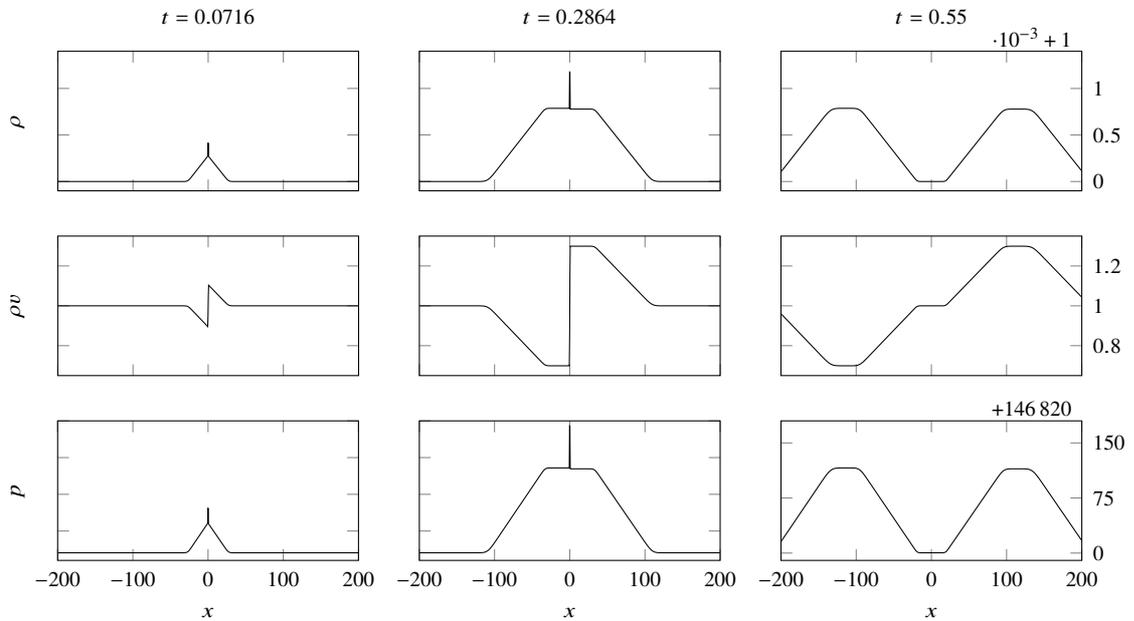
\begin{figure}
  \begin{tikzpicture}
  \begin{groupplot}[
        group style={group size=3 by 3,
            horizontal sep = .8 cm, 
            vertical sep = .6 cm,
            xlabels at=edge bottom,
            xticklabels at=edge bottom,
            yticklabels at=edge right}, 
          width = .355 \linewidth,
          height = .22 \linewidth,
          xlabel={$x$},
          xmin=-200,xmax=200,%ymin=-.1, ymax=1.1,
          every tick label/.append style={font=\scriptsize},
          title style={font=\scriptsize},
          label style={font=\scriptsize}
          ]
        \nextgroupplot[title={$t=0.0716$}, ylabel={$\rho$}, ymin=0.9999, ymax=1.0014]
        \addplot [] table [x index=0, y index=1] {input/psystem11jumpCPL2_2.dat};
        \nextgroupplot[title={$t=0.2864$}, ymin=0.9999, ymax=1.0014]
        \addplot [] table [x index=0, y index=1] {input/psystem11jumpCPL2_3.dat};
        \nextgroupplot[title={$t=0.55$}, ymin=0.9999, ymax=1.0014,
        scaled y ticks=manual:{$\cdot 10^{-3} + 1$}{
          \pgfmathparse{1000*(#1-1)}
        }, ytick={1, 1.0005, 1.001}]
        \addplot [] table [x index=0, y index=1] {input/psystem11jumpCPL2_4.dat};
        \nextgroupplot[ylabel={$\rho v$}, ymin=0.65, ymax=1.35]
        \addplot [] table [x index=0, y index=2] {input/psystem11jumpCPL2_2.dat};
        \nextgroupplot[ymin=0.65, ymax=1.35]
        \addplot [] table [x index=0, y index=2] {input/psystem11jumpCPL2_3.dat};
        \nextgroupplot[ymin=0.65, ymax=1.35]
        \addplot [] table [x index=0, y index=2] {input/psystem11jumpCPL2_4.dat};
        \nextgroupplot[ylabel={$p$}, ymin=146810, ymax=147000]
        \addplot [] table [x index=0, y index=3] {input/psystem11jumpCPL2_2.dat};
        \nextgroupplot[ymin=146810, ymax=147000]
        \addplot [] table [x index=0, y index=3] {input/psystem11jumpCPL2_3.dat};
        \nextgroupplot[ymin=146810, ymax=147000,
        scaled y ticks=manual:{$+ 146\,820$}{
          \pgfmathparse{(#1-146820)}
        }, ytick={146820, 146895, 146970}]
        \addplot [] table [x index=0, y index=3] {input/psystem11jumpCPL2_4.dat};
  \end{groupplot}
\end{tikzpicture}
\caption{Numerical solution for system~\eqref{eq:psystemrelaxation} in the relaxation limit using coupling approach~2 given by~\eqref{eq:psystemrelaxationcoupling}, $\beta_1=0$ and $\beta_2=1$. Artifacts at the interface ($x=0$) are visible. The coupled relaxed scheme on $1000$ mesh cells with $\text{CFL}=0.49$ has been used.}\label{fig:psystemcpl2}
\end{figure}
Figure~\ref{fig:psystemcpl2} shows the numerical solution in case of the second coupling approach given by~\eqref{eq:psystemrelaxationcoupling}, $\beta_1=0$ and $\beta_2=1$.
While this solution qualitatively coincides with the reference solution for $x\neq0$ it exhibits layer artifacts at the interface at the time instances $0.0716$ and $0.2864$.
\begin{figure}
  \begin{tikzpicture}[spy using outlines={rectangle, red, magnification=5,
    size=0.7cm, connect spies}]
  \begin{groupplot}[
        group style={group size=3 by 3,
            horizontal sep = .8 cm, 
            vertical sep = .6 cm,
            xlabels at=edge bottom,
            xticklabels at=edge bottom,
            yticklabels at=edge right}, 
          width = .355 \linewidth,
          height = .22 \linewidth,
          xlabel={$x$},
          xmin=-200,xmax=200,%ymin=-.1, ymax=1.1,
          every tick label/.append style={font=\scriptsize},
          title style={font=\scriptsize},
          label style={font=\scriptsize}
          ]
        \nextgroupplot[title={$t=0.0716$}, ylabel={$\rho$}, ymin=0.9999, ymax=1.0014]
        \addplot [] table [x index=0, y index=1] {input/psystem11jumpCPL3_2.dat};
        \nextgroupplot[title={$t=0.2864$}, ymin=0.9999, ymax=1.0014]
        \addplot [] table [x index=0, y index=1] {input/psystem11jumpCPL3_3.dat};
        \addplot [gray, dotted] table [x index=0, y index=1] {input/psystem11jumpCPL4_3.dat};
        \coordinate (spypoint) at (axis cs:0,1.00077); 
        \coordinate (magpoint) at (axis cs:100,1.001);
        \spy [red] on (spypoint) in node at (magpoint);
        \nextgroupplot[title={$t=0.55$}, ymin=0.9999, ymax=1.0014,
        scaled y ticks=manual:{$\cdot 10^{-3} + 1$}{
          \pgfmathparse{1000*(#1-1)}
        }, ytick={1, 1.0005, 1.001}]
        \addplot [] table [x index=0, y index=1] {input/psystem11jumpCPL3_4.dat};
        \nextgroupplot[ylabel={$\rho v$}, ymin=0.65, ymax=1.35]
        \addplot [] table [x index=0, y index=2] {input/psystem11jumpCPL3_2.dat};
        \nextgroupplot[ymin=0.65, ymax=1.35]
        \addplot [] table [x index=0, y index=2] {input/psystem11jumpCPL3_3.dat};
        \addplot [gray, dotted] table [x index=0, y index=2] {input/psystem11jumpCPL4_3.dat};
        \nextgroupplot[ymin=0.65, ymax=1.35]
        \addplot [] table [x index=0, y index=2] {input/psystem11jumpCPL3_4.dat};
        \nextgroupplot[ylabel={$p$}, ymin=146810, ymax=147000]
        \addplot [] table [x index=0, y index=3] {input/psystem11jumpCPL3_2.dat};
        \nextgroupplot[ymin=146810, ymax=147000]
        \addplot [] table [x index=0, y index=3] {input/psystem11jumpCPL3_3.dat};
        \addplot [gray, dotted] table [x index=0, y index=3] {input/psystem11jumpCPL4_3.dat};
        \nextgroupplot[ymin=146810, ymax=147000,
        scaled y ticks=manual:{$+ 146\,820$}{
          \pgfmathparse{(#1-146820)}
        }, ytick={146820, 146895, 146970}]
        \addplot [] table [x index=0, y index=3] {input/psystem11jumpCPL3_4.dat};
  \end{groupplot}
\end{tikzpicture}
\caption{Numerical solution for system~\eqref{eq:psystemrelaxation} in the relaxation limit using coupling approaches~3 (solid line) given by~\eqref{eq:psystemrelaxationcoupling} and $\beta_1=\beta_2=1$ and 4 (dotted line) given by~\eqref{eq:relaxationrhocondition},~\eqref{eq:rhovjumpagain},~\eqref{eq:v1condition},~\eqref{eq:nonlinearV2condition} and $\beta_1=\beta_2=1$. The coupled relaxed scheme on $1000$ mesh cells with $\text{CFL}=0.49$ has been used. Except for a jump of small magnitude occuring in approach 3 at the interface both solutions are indistinguishable.}\label{fig:psystemcpl3}
\end{figure}
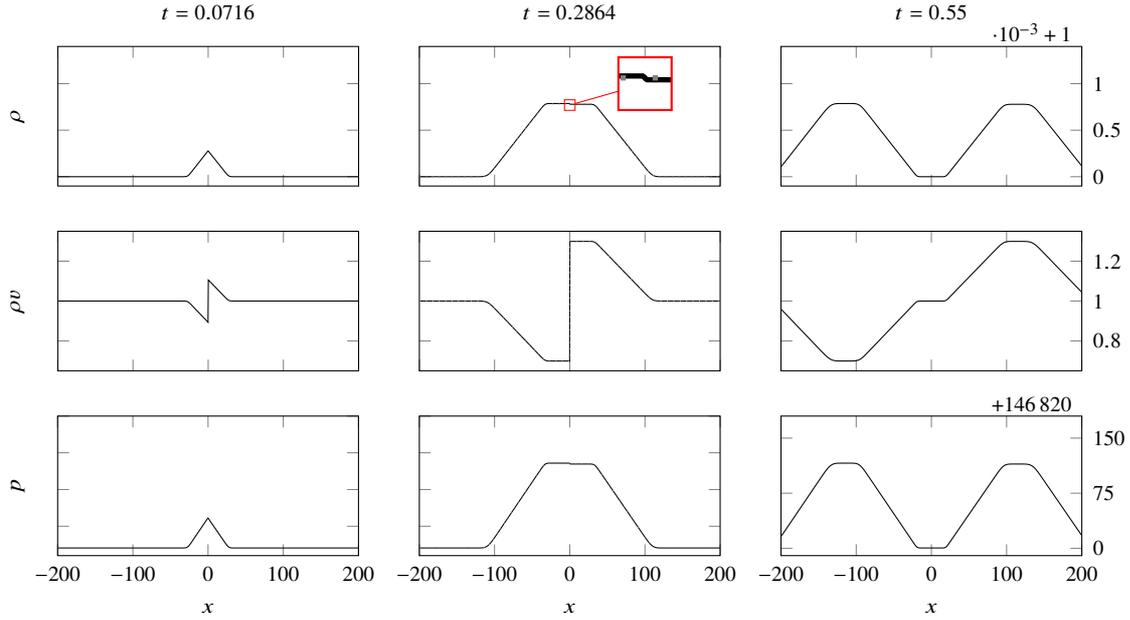
Eventually, numerical results for the third coupling approach given by~\eqref{eq:psystemrelaxationcoupling} and $\beta_1=\beta_2=1$ are shown in Figure~\ref{fig:psystemcpl3}. Using this condition the solution of the coupled relaxation system qualitatively matches the one of the original system~\eqref{eq:psystem} with coupling condition~\eqref{eq:psystemjumpcondition}. This indicates that the intrinsic contradiction in the relaxation limit between~\eqref{eq:rhovjumpagain} and~\eqref{eq:v1condition} in both coupling approaches one and two due to $\beta_1\neq \beta_2$ deteriorates the quality of the relaxation approach as an approximation of the original problem. We note hence, that consistency of the relaxation approach~\eqref{eq:psystemrelaxation} to system~\eqref{eq:psystem} as we have verified for our fourth coupling approach has not been necessary to achieve an accurate numerical solution. This is presumably due to the small magnitude of the additional term $\mathcal{E} (2 \rho_L v_L + \mathcal{E})/\rho_L$ in~\eqref{eq:nonlinearV2condition} that has been added in the fourth coupling approach, which compared to the magnitude of $V_{2R}$ and $V_{2L}$ is smaller by a factor of $10^{-6}$. A magnification at the interface of the solution corresponding to approach 3 reveals a jump of low magnitude not present in the reference solution. The solution for the fourth coupling approach given by~\eqref{eq:relaxationrhocondition},~\eqref{eq:rhovjumpagain},~\eqref{eq:v1condition},~\eqref{eq:nonlinearV2condition} and $\beta_1=\beta_2=1$ also shown in Figure~\ref{fig:psystemcpl3} and being qualitatively almost similar to the one of coupling approach 3 does not exhibit this artifact and therefore achieves the best fit to the reference solution.

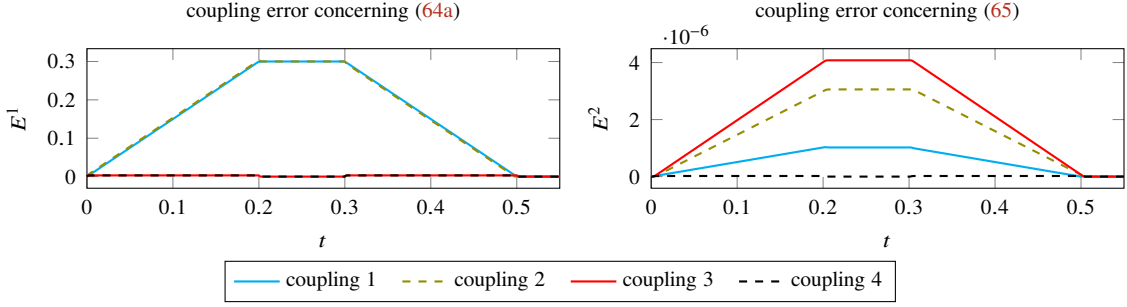
\begin{figure}
  \centering
\begin{tikzpicture}
  \begin{groupplot}[
        group style={group size=2 by 1,
            horizontal sep = 1.2 cm, 
            vertical sep = .6 cm,
            xticklabels at=edge bottom,
            xlabels at=edge bottom}, 
          width =.5 \linewidth,
          height = .22 \linewidth,
          xmin=0, xmax=0.55, %ymin=-1, ymax=410,
          xlabel={$t$},
          every tick label/.append style={font=\scriptsize},
          label style={font=\scriptsize},
          title style={font=\scriptsize},
          legend columns=-1,
          legend style={font=\scriptsize, fill=none, /tikz/every even column/.append style={column sep=.25cm}}
          ]
          \nextgroupplot[title={coupling error concerning~\eqref{eq:rhovjump}}, ylabel={$E^1$}]
          \addplot [color=cyan, thick] table [x index=0, y index=2] {input/psystem11jumpCPL_cplerror_1.dat};
        \addplot [color=olive, thick, dashed] table [x index=0, y index=2] {input/psystem11jumpCPL_cplerror_2.dat};
        \addplot [color=red, thick] table [x index=0, y index=2] {input/psystem11jumpCPL_cplerror_3.dat};
        \addplot [color=black, thick, dashed] table [x index=0, y index=2] {input/psystem11jumpCPL_cplerror_4.dat};
          \nextgroupplot[title={coupling error concerning~\eqref{eq:psystemcontinuousdensity}}, ylabel={$E^2$}, legend to name=leg:cplapproaches]
        \addplot [color=cyan, thick] table [x index=0, y index=1] {input/psystem11jumpCPL_cplerror_1.dat};
        \addplot [color=olive, thick, dashed] table [x index=0, y index=1] {input/psystem11jumpCPL_cplerror_2.dat};
        \addplot [color=red, thick] table [x index=0, y index=1] {input/psystem11jumpCPL_cplerror_3.dat};
        \addplot [color=black, thick, dashed] table [x index=0, y index=1] {input/psystem11jumpCPL_cplerror_4.dat};
        \legend{coupling 1, coupling 2, coupling 3, coupling 4}
  \end{groupplot}
\end{tikzpicture}
\pgfplotslegendfromname{leg:cplapproaches}
\caption{Coupling errors $E^1$ (left) and $E^2$ (right) defined in~\eqref{eq:couplingerrors} over the time interval $[0,0.55]$ for the four considered coupling approaches. Errors for numerical solutions computed by the central scheme on 1000 mesh cells using $\text{CFL}=0.49$ are considered. Both coupling approaches 3 and 4 achieve significantly lower errors than approaches 1 and 2 concerning~\eqref{eq:rhovjump}. Approach 4 reduces the error concerning~\eqref{eq:psystemcontinuousdensity} over approaches 1--3.}\label{fig:couplingerrors}
\end{figure}

For an insight into the approximation quality of the central scheme at the interface we consider the coupling errors with respect to the original coupling conditions~\eqref{eq:rhovjump} and~\eqref{eq:psystemcontinuousdensity} defined as
\begin{equation}\label{eq:couplingerrors}
  E^1(t^n) = \left| (U_{-1}^n - U_{0}^n)_2 - \mathcal{E} \right|, \qquad E^2(t^n) = \left|(U_{-1}^n - U_{0}^n)_1 \right|
\end{equation}
Using the numerical solutions by the central scheme over 1000 mesh cells and employing Courant number $\text{CFL}=0.49$ we present these errors over time in Figure~\ref{fig:couplingerrors}. In case of the error concerning the density condition~\eqref{eq:psystemcontinuousdensity} the temporal errors follow the monotony of the momentum outtake $\mathcal{E}$. Approaches 4 and 1 here yield lower errors than approaches 2 and 3. Concerning the momentum condition~\eqref{eq:rhovjump} both approaches 1 and 2 yield coupling errors in the magnitude of the momentum outtake and are clearly outperformed by approaches 3 and 4 both achieving similarly small errors.  

\begin{table}\scriptsize
  \centering
  \caption{Mesh convergence of the coupling error $L^1(E^1)$ over time for the considered coupling approaches. First order error convergence for approaches 3 and 4 is indicated.}\label{tab:couplingconvergence}
\begin{tabular}{r l l l l l l l l} \toprule
  cells & coupling 1 & EOC & coupling 2 & EOC & coupling 3 & EOC & coupling 4 & EOC \\ \midrule
100 & \(8.999 \times 10^{-2}\) & & \(9.155 \times 10^{-2}\) & &
\(1.278 \times 10^{-2}\) & & \(1.278 \times 10^{-2}\) & \\
200 & \(8.999 \times 10^{-2}\) & 0.00 & \(9.040 \times 10^{-2}\) & 0.02
& \(6.325 \times 10^{-3}\) & 1.01 & \(6.324 \times 10^{-3}\) & 1.01 \\
400 & \(8.999 \times 10^{-2}\) & 0.00 & \(9.011 \times 10^{-2}\) & 0.00
& \(3.147 \times 10^{-3}\) & 1.01 & \(3.146 \times 10^{-3}\) & 1.01 \\
800 & \(8.999 \times 10^{-2}\) & 0.00 & \(9.004 \times 10^{-2}\) & 0.00
& \(1.569 \times 10^{-3}\) & 1.00 & \(1.569 \times 10^{-3}\) & 1.00 \\
1600 & \(8.999 \times 10^{-2}\) & 0.00 & \(9.003 \times 10^{-2}\) & 0.00
& \(7.838 \times 10^{-4}\) & 1.00 & \(7.837 \times 10^{-4}\) & 1.00 \\ \bottomrule
\end{tabular}
\end{table}

To investigate the scaling of the coupling errors with the mesh resolution we compute numerical solutions for increasing numbers of mesh cells and consider the errors~\eqref{eq:couplingerrors} in the discrete $L^1$ norm over time, i.e.,
\[
  L^1(E^i)= \sum_{k=0}^{k_T} \Delta t E^i(t^k),\qquad k_T = \lceil 0.55/\Delta t\rceil, \qquad i=1,2.
\]
Table~\ref{tab:couplingconvergence} shows the error $L^1(E^1)$ for the considered coupling approaches after mesh refinement along with the corresponding experimental order of convergence (EOC)\footnote{The EOC is computed by the formula $\text{EOC}=\log_2(E^1/E^2)$ with $E^1$ and $E^2$ denoting the error in two consecutive lines of the table.}. In case of coupling approaches 1 and 2 the errors are not significantly decreasing when refining the mesh. Conversely mesh convergence of order one can be clearly observed for the coupling approaches 3 and 4. Moreover, a minor advantage of coupling approach 4 over coupling approach 3 can be seen from the errors in Table~\ref{tab:couplingconvergence}.

\begin{table}\scriptsize
  \centering
  \caption{Mesh convergence of the coupling error $L^1(E^2)$ over time for the considered coupling approaches. First order error convergence is only indicated for approach 4.}\label{tab:couplingconvergenceE2}
\begin{tabular}{r l l l l l l l l} \toprule
  cells & coupling 1 & EOC & coupling 2 & EOC & coupling 3 & EOC & coupling 4 & EOC \\ \midrule
100 & \(3.131 \times 10^{-7}\) & & \(9.173 \times 10^{-7}\) & &
\(1.224 \times 10^{-6}\) & & \(8.685 \times 10^{-8}\) & \\
200 & \(3.083 \times 10^{-7}\) & 0.02 & \(9.179 \times 10^{-7}\) & 0.00
& \(1.224 \times 10^{-6}\) & 0.00 & \(4.307 \times 10^{-8}\) & 1.01 \\
400 & \(3.070 \times 10^{-7}\) & 0.01 & \(9.179 \times 10^{-7}\) & 0.00
& \(1.224 \times 10^{-6}\) & 0.00 & \(2.143 \times 10^{-8}\) & 1.01 \\
800 & \(3.067 \times 10^{-7}\) & 0.00 & \(9.179 \times 10^{-7}\) & 0.00
& \(1.224 \times 10^{-6}\) & 0.00 & \(1.069 \times 10^{-8}\) & 1.00 \\
1600 & \(3.066 \times 10^{-7}\) & 0.00 & \(9.179 \times 10^{-7}\) & 0.00
& \(1.224 \times 10^{-6}\) & 0.00 & \(5.336 \times 10^{-9}\) & 1.00 \\ \bottomrule
\end{tabular}
\end{table}
The error $L^1(E^2)$ after mesh refinement is presented in Table~\ref{tab:couplingconvergenceE2}. In case of the linear coupling approaches no significant variation with the mesh resolution can be observed while the errors are of small magnitude. Convergence with respect to the mesh resolution is only observed in case of coupling approach 4 and consistency with the original problem.

\section{Conclusion}
We have presented a novel framework to handle systems of conservation laws coupled at an interface using a relaxation approach. When systems of conservation laws are coupled usually information on the Lax-curves of the underlying systems as well as complex and error-prone computations are required to identify suitable coupling data. The relaxation approach avoids this complication by replacing the nonlinear systems of conservation laws by systems of linear balance laws. We have introduced Riemann solvers for the new approach, discussed their well-posedness and proposed a finite volume scheme that can be generally applied to coupled hyperbolic problems.   

The approach requires a suitable adjustment of the original coupling condition to account for the modified systems. Considering the $p$-system in a case study we have illustrated how the choice of this condition affects the numerical solutions. While accurate results could be obtained using simple linear ad-hoc conditions, best results that also exhibit convergence at the interface have been achieved constructing a coupling condition both leading to a well-defined Riemann solver and being in the relaxation limit consistent with the original coupling condition. Motivated by these results we suggest the same strategy for the design of the relaxation coupling condition also for different coupled  hyperbolic systems.

\section*{Statements and Declarations}
\textbf{Compliance with Ethical Standards} The authors state that they comply with the ethical standards.\\[0.2cm]
\textbf{Funding} The authors thank the Deutsche Forschungsgemeinschaft (DFG, German Research Foundation) for the financial support through 320021702/GRK2326,  333849990/IRTG-2379, B04, B05 and B06 of 442047500/SFB1481, HE5386/18-1,19-2,22-1,23-1,25-1, ERS SFDdM035 and under Germany’s Excellence Strategy EXC-2023 Internet of Production 390621612 and under the Excellence Strategy of the Federal Government and the Länder. Support through the EU DATAHYKING is also acknowledged. 
\\[0.2cm]
\textbf{Conflict of Interest} On behalf of all authors, the corresponding author states that there is no conflict of interest.\\[0.2cm]
\textbf{Ethical Approval} Not applicable.
 
\bibliographystyle{abbrvurl} 
\bibliography{references.bib}
% \bibliography{/Users/niklas/org/bib/global.bib}
\end{document}

%% file: input/wave-structure.tikz
\begin{tikzpicture}[x=\linewidth/25,y=\linewidth/25]%\draw[color=white][->, thin] (0,-1.8) -- (0,2.9);
\draw[->, thin] (-8,0) -- (8,0) node[below] {$x$};
% \node[label=below:$-2$] at (-10, 0){$|$};
% \node[label=below:$-1$] at (-5, 0){$|$};
\node[label=below:{$0$}] at (0, 0){};
% \node[label=below:$1$] at (5, 0){$|$};
% \node[label=below:$2$] at (10, 0){$|$};
\draw[->, thin, dashed] (0, 0) -- (0, 5) node[label=above:{$\Psiq[Q_R, Q_L]=0$}] {};
\node[left] at (0, 4.5){$t$};
\node at (-5, .8){$Q_0^-$};
\draw[->] (-5, 1.3) arc (160:105:4); % (160:130:4)
\node at (-1.5, 4){$Q_R$};
\node at (5, .8){$Q_0^+$};
\draw[->] (5, 1.3) arc (20:75:4); % (20:50:4)
\node at (1.5, 4){$Q_L$};
\draw[-] (0, 0) -- (-7, 3) node[left] {$\mathcal{L}_{\lambda_1}(Q_0^-)$};
\draw[-] (0, 0) -- (-4, 4.7) node[left] {$\mathcal{L}_{\lambda_n}(Q_0^-)$};
\draw[-] (0, 0) -- (4, 4.7) node[right] {$\mathcal{L}_{\lambda_{n+1}}(Q_0^+)$};
\draw[-] (0, 0) -- (7, 3) node[right] {$\mathcal{L}_{\lambda_{2n}}(Q_0^+)$};
\end{tikzpicture}

%% file: input/one-one-num-coupling.tikz
\begin{tikzpicture}[x=\linewidth/25,y=\linewidth/25]%\draw[color=white][->, thin] (0,-1.8) -- (0,2.9);
\draw[->, thin] (-12,0) -- (12,0) node[right] {$x$};
\node[label=below:$x_{-5/2}$] at (-10, 0){$|$};
\node[label=below:$I_{-2}$] at (-7.5, 0){};
\node[label=above:$Q_{-2}$] at (-7.5, 0){};
\node[label=below:$x_{-3/2}$] at (-5, 0){$|$};
\node[label=below:$I_{-1}$] at (-2.5, 0){};
\node[label=above:$Q_{-1}$] at (-2.5, 0){};
\node[label=above:$ \color{gray}Q_{R}$] at (-0.75, 0){};
\node[label=below:{$x_{-1/2}=0$}] at (0, 0){$|$};
\node[label=above:$ \color{gray}Q_{L}$] at (0.75, 0){};
\node[label=below:$I_0$] at (2.5, 0){};
\node[label=above:$Q_0$] at (2.5, 0){};
\node[label=below:$x_{1/2}$] at (5, 0){$|$};
\node[label=below:$I_1$] at (7.5, 0){};
\node[label=above:$Q_1$] at (7.5, 0){};
\node[label=below:$x_{3/2}$] at (10, 0){$|$};
\draw[-, thin, dashed] (0, 0) -- (0, 4) node[label=above:{$\Psi[Q_R, Q_L]=0$}] {};
\coordinate (A) at (-5,1);
\node[label=above:{$\begin{aligned}
  \ddt Q + S_1 \, \ddx Q = \frac{1}{\varepsilon}
                      \begin{pmatrix}
                        0 \\
                        F_1(U) - V
                      \end{pmatrix}
                     \end{aligned}$}] at (-6, 1.5) {};
\node[label=above:{$\begin{aligned}
                      \ddt Q + S_2 \, \ddx Q = \frac{1}{\varepsilon}
                      \begin{pmatrix}
                        0 \\
                        F_2(U) - V
                      \end{pmatrix}
                     \end{aligned}$}] at (6, 1.5) {};
\end{tikzpicture}